\documentclass[10pt,reqno]{amsart}
\usepackage[margin=1in]{geometry}

\usepackage{amsmath,amsthm,wasysym,bbold,tensor,mdframed}
\usepackage{amssymb}
\usepackage{amsfonts}
\usepackage{mathscinet}
\usepackage[dvipsnames]{xcolor}
\usepackage{tikz-cd}
\usepackage{tikz,keyval,pgfsys,graphicx}
\usepackage{mathtools, enumerate, combelow}
\usepackage[shortlabels]{enumitem}

\usetikzlibrary{arrows, decorations.pathreplacing}
\usepgflibrary{arrows}
\newtheorem{thm}{\normalfont\scshape Theorem}[section]

\newtheorem{prop}[thm]{\normalfont\scshape Proposition}
\newtheorem{lem}[thm]{\normalfont\scshape Lemma}
\newtheorem{cor}[thm]{\normalfont\scshape Corollary}

\theoremstyle{definition}
\newtheorem{defn}[thm]{\normalfont\scshape Definition}

\theoremstyle{remark}
\newtheorem{rem}[thm]{Remark}
\theoremstyle{remark}

\allowdisplaybreaks

\DeclareMathOperator{\TT}{\mathcal{T}}

\DeclareMathOperator{\DD}{\mathbb{D}}

\begin{document}

\binoppenalty=10000
\relpenalty=10000

\numberwithin{equation}{section}

\newcommand{\uuu}{\mathfrak{u}}
\newcommand{\ccc}{\mathfrak{c}}
\newcommand{\qqq}{\mathfrak{q}}
\newcommand{\ddd}{\mathfrak{d}}
\newcommand{\QQ}{\mathbb{Q}}
\newcommand{\ZZ}{\mathbb{Z}}
\newcommand{\Hilb}{\mathrm{Hilb}}
\newcommand{\CC}{\mathbb{C}}
\newcommand{\PP}{\mathcal{P}}
\newcommand{\Hom}{\mathrm{Hom}}
\newcommand{\Rep}{\mathrm{Rep}}
\newcommand{\MM}{\mathfrak{M}}
\newcommand{\gl}{\mathfrak{gl}}
\newcommand{\ppp}{\mathfrak{p}}
\newcommand{\VV}{\mathcal{V}}
\newcommand{\NN}{\mathbb{N}}
\newcommand{\OO}{\mathcal{O}}
\newcommand{\HH}{\mathcal{H\kern-.44em H}}
\newcommand{\git}{/\kern-.35em/}
\newcommand{\KK}{\mathbb{K}}
\newcommand{\Proj}{\mathrm{Proj}}
\newcommand{\quot}{\mathrm{quot}}
\newcommand{\core}{\mathrm{core}}
\newcommand{\Sym}{\mathrm{Sym}}
\newcommand{\FF}{\mathbb{F}}
\newcommand{\Span}[1]{\mathrm{span} \left\{#1 \right\}}
\newcommand{\UTor}{U_{\qqq,\ddd}(\ddot{\mathfrak{sl}}_r)}
\newcommand{\Sss}{\mathcal{S}}

\title[Five-Term Relations for wreath Macdonald polynomials]{Five-Term Relations for wreath Macdonald polynomials and tableau formulas for Pieri coefficients}
\author{Marino Romero} 
\email{marino.romero@univie.ac.at}
\address{Fakult\"{a}t f\"{u}r Mathematik, Universit\"{a}t Wien, Vienna, Austria}

\author{Joshua Jeishing Wen}
\email{joshua.jeishing.wen@univie.ac.at}
\address{Fakult\"{a}t f\"{u}r Mathematik, Universit\"{a}t Wien, Vienna, Austria}


\begin{abstract}
    We present a variety of new identities involving operators in the theory of wreath Macdonald polynomials. One such family of identities gives five-term relations, analogous to the one given by Garsia and Mellit for the modified Macdonald polynomials. As a consequence, we generate tableau formulas for wreath Macdonald Pieri coefficients, which give an incredibly quick way of computing their monomial expansions.
\end{abstract}

\maketitle

\section{Introduction}

Defined by Haiman \cite{Haiman}, the \textit{wreath Macdonald polynomials} generalize the modified Macdonald polynomials from the symmetric group $\Sigma_n$ to its wreath product $\Gamma_n:= \ZZ/r\ZZ\wr\Sigma_n$.
Just as one can study the representation theory of the symmetric groups $\{\Sigma_n\}$ using the ring of symmetric functions $\Lambda$, for fixed $r$, one can model the representation theory of $\{\Gamma_n\}$ via the ring of multisymmetric functions $\Lambda^{\otimes r}$ \cite{Mac}.
The wreath Macdonald polynomials are distinguished elements of the deformed ring $\CC(q,t)\otimes\Lambda^{\otimes r}$.
However, it may be naive to simply treat wreath Macdonald theory as a partially symmetric analogue of Macdonald theory.
Haiman's original definition was motivated by geometry.
Just as his celebrated proof of the Macdonald positivity conjecture \cite{HaimanHilb} showed that Macdonald polynomials are intimately tied to the geometry of Hilbert schemes of points on $\CC^2$, Haiman conjectured that the wreath Macdonald polynomials should be related to certain Nakajima quiver varieties for the cyclic quiver.
This conjecture was realized by Bezrukavnikov--Finkelberg \cite{BezFink}, who proved a wreath analogue of Macdonald positivity.

Functional aspects of wreath Macdonald polynomials did not receive much attention until recently, and it was not evident which properties and identities had wreath analogues. 
For instance, the usual Macdonald polynomials are solutions to $q$-difference equations; wreath analogues of this were finally established by the second author with Orr and Shimozono \cite{WreathEigen, OSW}.
More relevant to this paper, Tesler's identity and the five-term relation of Garsia and Mellit \cite{Garsia-Mellit} provide two of the most powerful tools in developing symmetric function identities related to modified Macdonald polynomials, even finding applications in geometry. From here, one can give a number of consequences, such as proving tableau formulas for Pieri coefficients \cite{BergeronHaiman} or giving identities for Theta operators \cite{DadderioRomero}. The recursion for Pieri coefficients implied by the five-term relation gives a fast, if not the fastest, way of computing the monomial expansion for modified Macdonald polynomials.
A version of Tesler's identity was recently proved for wreath Macdonald polynomials in \cite{RW}, where a number of plethystic operators were also developed. 

In this paper, we derive consequences of the wreath Tesler identity and push the functional theory of wreath Macdonald polynomials further into maturation.
First, we revisit the plethystic calculus on multisymmetric functions, giving a concise way of notating vector (or matrix) plethysms. We then demonstrate how to utilize this device by reproving reciprocity identities from the wreath Tesler identity.
Next, we give several families of five-term relations for operators on wreath Macdonald polynomials, and as a consequence, give formulas for computing Pieri coefficients in terms of tableaux. 
To clarify, our Pieri rules are not the wreath analogues of those found in \cite[Chapter VI]{Mac}; it appears that such a wreath analogue is quite complicated (cf. \cite{WreathOrth}).
Nonetheless, these recursive formulas are incredibly powerful, allowing one to compute many more examples of wreath Macdonald polynomials than was previously possible. 

We also introduce a wreath version for the Theta operators of D'Adderio, Iraci, and Vanden Wyngaerd \cite{Theta}, and show how they relate to the $\mathbb{D}$ operators in \cite{RW}. In the case of modified Macdonald polynomials, Theta operators provide a framework for producing many useful identities \cite{DadderioRomero}. They were originally introduced to give a compositional refinement of the Delta Conjecture of Haglund, Remmel, and Wilson \cite{DeltaConjecture}. This refinement was proved by D'Adderio and Mellit \cite{DadderioMellit}. Theta operators also conjecturally produce the Frobenius characteristic for the space of coinvariants in two sets of commuting and two sets of anticommuting variables \cite{Theta}.
We predict that wreath Theta operators will also give a powerful tool for dealing with multisymmetric function identities related to wreath Macdonald polynomials. Here, we will only look at an identity that explains the relation between wreath Theta operators and the action of the quantum toroidal algebra on multisymmetric functions.

\section{Combinatorial definitions}
Before we begin, let us first fix $r>2$. 
It will be useful to think of the symmetric group $\Sigma_r $ as acting on the elements of $\ZZ/r\ZZ$.  

\subsection{Partitions and characters}\label{PartNot}
A partition $\lambda=(\lambda_1,\lambda_2,\dots, \lambda_\ell)$ is a weakly decreasing sequence of positive integers; $\ell(\lambda) = \ell$ is its length; and $|\lambda| = \lambda_1+\cdots +\lambda_\ell$ its size, in which case we write $\lambda \vdash |\lambda|$.
The \textit{transposed} partition ${}^t\lambda$ has parts given by
\[
{}^t\lambda_i \coloneqq\#\left\{ k\, |\, m_k\ge i \right\}.
\]
For $\lambda,\mu\vdash n$, $\lambda\le\mu$ will denote dominance order, meaning that for all $i$,
\[
\lambda_1+\cdots+\lambda_i\le\mu_1+\cdots\mu_i.
\]

To $\lambda$, we associate its \textit{Young diagram}, which we draw following the French convention.
For example, $\lambda=(4,2,2)$ has the Young diagram
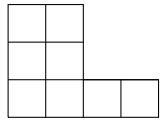
\vspace{.1in}

\centerline{\begin{tikzpicture}[scale=.5]
\draw (0,0)--(4,0)--(4,1)--(2,1)--(2,3)--(0,3)--(0,0);;
\draw (0,1)--(4,1);;
\draw (0,2)--(2,2);;
\draw (1,0)--(1,3);;
\draw (2,0)--(2,1);;
\draw (3,0)--(3,1);;
\end{tikzpicture}.}

\vspace{.1in}

\noindent We assign Cartesian coordinates to the cells in the Young diagram so that the box at the lower left corner has coordinates $(0,0)$.
Thus, $(i,j)\in\lambda$ is in the $i+1$-st column and $j+1$-st row. The character of $\square = (i,j)$ is given by
\[
\chi_\square \coloneqq q^i t^j.
\]

\begin{defn}
    For a given paritition $\lambda$, let
\begin{align*}
    B_\lambda \coloneqq \sum_{\square \in \lambda } 
    \chi_\square && \text{ and } && D_\lambda \coloneqq (1-q)(1-t) B_\lambda -1.
\end{align*}
\end{defn}
Given another partition $\mu$, $\mu/\lambda$ is the diagram formed by removing the cells of $\lambda$ from $\mu$. We extend any notation involving characters to skew diagrams, so that, for instance, $B_{\mu/\lambda}$ is the sum of characters for cells in $\mu/\lambda$. 

The \textit{content} of $\square=(i,j)$ is the quantity
\[
c_\square:= j-i.
\]
Graphically, it is the SW-NE diagonal that the box lies on.
We call the class of $c_\square$ modulo $r$ its \textit{color} and denote its value by $\bar{c}_\square$. The number of cells of color $i$ is denoted by $d_i(\lambda)$. \\

\begin{defn}
    If $\mu$ is attained from $\lambda$ by adding $k$ cells of each color, then we will write $\lambda \subset_k \mu$.
\end{defn}
We will see later that in light of Proposition \ref{CoreQuotDec}, if $\lambda \subset_k \mu$, then $\core(\lambda) = \core(\mu)$.
\\

An addable corner of $\lambda$ is a cell $\square \not \in \lambda$ for which $\lambda \cup \square$ is a partition. A removable corner is a cell $\square \in \lambda$ for which $A / \square$ is a partition.
Let $A_i(\lambda)$ denote the addable cells of $\lambda$ with color $i$, and similarly, let $R_i(\lambda)$ be the removable corners of color $i$.

\subsection{Characters and colors}
For a series $S = \sum_{a,b} c_{a,b} q^a t^b$ in $q$ and $t$, we let 
\[
S^{(i)} \coloneqq \sum_{i \equiv b-a \text{ mod } r} c_{a,b} q^a t^b
\]
be the part of color $i$. We will also set 
\begin{align*}
S_\ast = S\Big|_{(q,t) \rightarrow (q^{-1},t^{-1})}
&&
\text{ and }
&&
S_\ast^{(i)} \coloneqq \left(S^{(i)}\right)_\ast = S^{(i)} \Big|_{(q,t) \rightarrow (q^{-1},t^{-1})}.
\end{align*}

With this notation, we have
\begin{align*}
B_{\lambda}^{(i)} = \sum_{ \substack{ \square \in \lambda \\ \bar{c}_\square = i} }
\chi_\square && \text{ and } && 
D^{(i)}_\lambda = 
-t B_{\lambda}^{(i-1)} +(1+qt)B_{\lambda}^{(i)} - q B_{\lambda}^{(i+1)} - \delta_{i,0}.
\end{align*}
One can show that 
\[
D^{(i)}_\lambda = qt\sum_{\square \in R_i(\lambda)} \chi_\square
- \sum_{\square \in A_i(\lambda)} \chi_\square  . 
\]

\subsection{Multipartitions and the Young-Maya correspondence}
An $r$-tuple 
$
\vec{\lambda}=(\lambda^0,\ldots, \lambda^{r-1})
$
of partitions
has size given by 
\[
|\vec{\lambda}|:=\sum_{i=0}^{r-1}|\lambda^i|,
\]
in which case we will write $\vec{\lambda} \vdash |\vec{\lambda}|$. 
To avoid confusion, we will index the coordinates of the tuple by superscripts.
It will be useful to view the indexing set of the tuple as $\ZZ/r\ZZ$, so that $\lambda^i = \lambda^j$ if $i \equiv j\text{ mod }r$.

Concerning a single partition $\lambda$, a \textit{ribbon} of length $r$ is a set of $r$ contiguous boxes on the outer (NE) edge of $\lambda$ whose removal leaves behind another partition. (This is often referred to as a rim-hook in other contexts.)
An $r$-core partition is partition with no ribbons of length $r$.

\subsubsection{Maya diagrams}
A Maya diagram is a function $m:\ZZ\rightarrow\{\pm 1\}$ such that
\[
m(n)=
\begin{cases}
-1 & n\gg0\\
1 & n\ll 0
\end{cases}.
\]
We can visualize $m$ as a string of black and white beads indexed by $\ZZ$: the bead at position $n$ is black if $m(n)=1$ and white if $m(n)=-1$.
These beads will be arranged horizontally with the index increasing towards the \textit{left}.
Between indices $0$ and $-1$, we draw a notch and call it the \textit{central line}.
The \textit{charge} $c(m)$ is the difference in the number of discrepancies:
\[
c(m)=\#\left\{ n\ge 0 \, |\, m(n)=1 \right\}-\#\left\{n<0\, |\, m(n)= -1\right\}.
\]
The \textit{vacuum diagram} is the Maya diagram with only white beads left of the central line and only black beads right of the central line.

To a partition $\lambda$ we associate a Maya diagram $m(\lambda)$ via its \textit{edge sequence}---we call this the \textit{Young--Maya correspondence}.
Namely, tilt $\lambda$ by $45$ degrees counterclockwise and draw the content lines (which are now vertical).
Label by $n$ the gap between content lines $n$ and $n+1$.
Within the gap labeled by $n$, the outer (formerly NE) edge of $\lambda$ either has slope $1$ or $-1$; and $m(\lambda)(n)$ takes the value of the slope of that segment.
Outside of $\lambda$, we assign the default limit assignments (or use the axes as part of the outer edge of $\lambda$).
For example, Figure \ref{fig:maya-diagram} gives the Young--Maya correspondence for $\lambda=(4,2,2)$.

\begin{figure}
\centerline{
\begin{tikzpicture}[scale=.5]
\draw[thick] (-2.5,5)--(-0.5,7)--(1.5,5)--(3.5,7)--(4.5,6);;
\draw (-1.5,4)--(0.5,6);;
\draw (-.5,3)--(3.5,7);;
\draw (-2.5,5)--(.5,2)--(4.5,6);;
\draw (-1.5,6)--(1.5,3);;
\draw (1.5,5)--(2.5,4);;
\draw (2.5,6)--(3.5,5);;
\draw (3.5,7)--(4.5,6);;
\draw (-4,0) node {$\cdots$};;
\draw (-3,0) node {3};;
\draw (-3,1) circle (5pt);;
\draw (-2,0) node {2};;
\draw[fill=black] (-2,1) circle (5pt);;
\draw (-1,0) node {1};;
\draw[fill=black] (-1,1) circle (5pt);;
\draw (0,0) node {0};;
\draw (0,1) circle (5pt);;
\draw (.5,.5)--(.5,1.5);;
\draw (1,0) node {-1};;
\draw (1,1) circle (5pt);;
\draw (2,0) node {-2};;
\draw[fill=black] (2,1) circle (5pt);;
\draw (3,0) node {-3};;
\draw[fill=black] (3,1) circle (5pt);;
\draw (4,0) node {-4};;
\draw (4,1) circle (5pt);;
\draw (5,0) node {-5};;
\draw[fill=black] (5,1) circle (5pt);;
\draw (6,0) node {$\cdots$};;
\draw[dashed] (1.5,2)--(1.5,9);;
\draw[dashed] (2.5,2)--(2.5,9);;
\draw[dashed] (3.5,2)--(3.5,9);;
\draw[dashed] (4.5,2)--(4.5,9);;
\draw[dashed] (0.5,2)--(0.5,9);;
\draw[dashed] (-0.5,2)--(-0.5,9);;
\draw[dashed] (-1.5,2)--(-1.5,9);;
\draw[dashed] (-2.5,2)--(-2.5,9);;
\end{tikzpicture}
} 
\caption{The Maya diagram for the partition $(4,2,2)$.}
\label{fig:maya-diagram}
\end{figure}
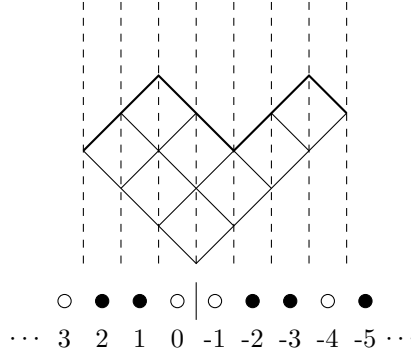

\vspace{.1in}

\begin{prop}\label{YoungMaya}
The Young--Maya correspondence is a bijection between partitions and charge $0$ Maya diagrams.
\end{prop}

Given $\lambda$ and $0\le i\le r-1$, we take the following subdiagrams of $m(\lambda)$:
\[
m_i(\lambda)(n)\coloneqq m(\lambda)(i+nr)
\]
These have a possibly nonzero charge $c_i$.
If we draw a notch in $m_i(\lambda)$ immediately right of the bead for $c_i$ and treat it as the central line, we obtain a charge zero Maya diagram.
By Proposition \ref{YoungMaya}, this corresponds to some partition $\lambda^i$, and we set
\[
\quot(\lambda)\coloneqq(\lambda^0,\ldots, \lambda^{r-1}).
\]
In Figure \ref{fig:quotient}, we show that for $r=3$ and $\lambda=(4,2,2)$, $\quot(\lambda)=( \varnothing, \varnothing, (1,1))$. Here, we have drawn the central line coming from $m(\lambda)$ using solid notches and the $c_i$-shifted central lines using dashed notches.

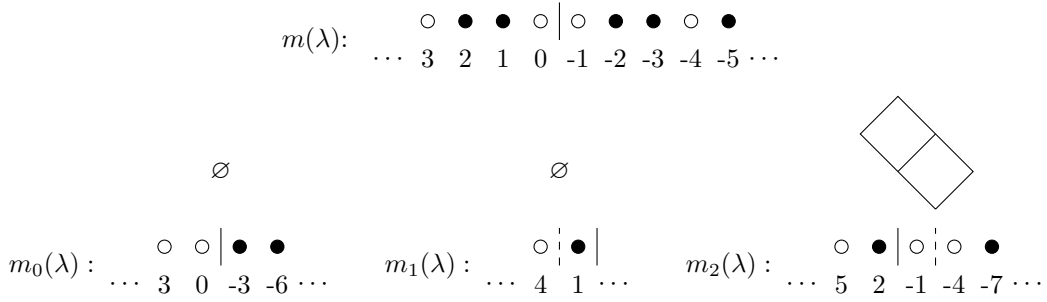
\begin{figure}[h]
\centerline{
\begin{tikzpicture}[scale=.5]
\draw (-8,.5) node {$m(\lambda)$:};;
\draw (-6,0) node {$\cdots$};;
\draw (-5,0) node {3};;
\draw (-5,1) circle (5pt);;
\draw (-4,0) node {2};;
\draw[fill=black] (-4,1) circle (5pt);;
\draw (-3,0) node {1};;
\draw[fill=black] (-3,1) circle (5pt);;
\draw (-2,0) node {0};;
\draw (-2,1) circle (5pt);;
\draw (-1.5,.5)--(-1.5,1.5);;
\draw (-1,0) node {-1};;
\draw (-1,1) circle (5pt);;
\draw (0,0) node {-2};;
\draw[fill=black] (0,1) circle (5pt);;
\draw (1,0) node {-3};;
\draw[fill=black] (1,1) circle (5pt);;
\draw (2,0) node {-4};;
\draw (2,1) circle (5pt);;
\draw (3,0) node {-5};;
\draw[fill=black] (3,1) circle (5pt);;
\draw (4,0) node {$\cdots$};;
\draw (-15,-5.5) node {$m_0(\lambda):$};;
\draw (-13,-6) node {$\cdots$};;
\draw (-12,-6) node {3};;
\draw (-12,-5) circle (5pt);;
\draw (-11,-6) node {0};;
\draw (-11,-5) circle (5pt);;
\draw (-10,-6) node {-3};;
\draw[fill=black] (-10,-5) circle (5pt);;
\draw (-10.5,-3) node {$\varnothing$};;
\draw (-9,-6) node {-6};;
\draw[fill=black] (-9,-5) circle (5pt);;
\draw (-8,-6) node {$\cdots$};;
\draw (-10.5,-5.5)--(-10.5,-4.5);;
\draw (-5,-5.5) node {$m_1(\lambda):$};;
\draw (-3,-6) node {$\cdots$};;
\draw (-2,-6) node {4};;
\draw[dashed] (-1.5,-5.5)--(-1.5,-4.5);;
\draw (-1.5, -3) node {$\varnothing$};;
\draw (-2,-5) circle (5pt);;
\draw (-1,-6) node {1};;
\draw[fill=black] (-1,-5) circle (5pt);;
\draw (-.5,-5.5)--(-.5,-4.5);;
\draw (0,-6) node {$\cdots$};;
\draw (3, -5.5) node {$m_2(\lambda):$};;
\draw (5,-6) node {$\cdots$};;
\draw (6,-6) node {5};;
\draw (6,-5) circle (5pt);;
\draw (7,-6) node {2};;
\draw[fill=black] (7,-5) circle (5pt);;
\draw (7.5,-5.5)--(7.5,-4.5);;
\draw (8.5, -4)--(9.5,-3)--(7.5,-1)--(6.5,-2)--(8.5,-4);;
\draw (7.5,-3)--(8.5,-2);;
\draw (8,-6) node {-1};;
\draw[dashed] (8.5,-5.5)--(8.5,-4.5);;
\draw (8,-5) circle (5pt);;
\draw (9,-6) node {-4};;
\draw (9,-5) circle (5pt);;
\draw (10,-6) node {-7};;
\draw[fill=black] (10,-5) circle (5pt);;
\draw (11,-6) node {$\cdots$};;
\end{tikzpicture}
}
\caption{The quotient decomposition for $\lambda = (4,2,2)$ when $r=3$.}
\label{fig:quotient}
\end{figure}

The $r$-core $\core(\lambda)$ is the result of removing ribbons of length $r$ from $\lambda$ until it is no longer possible.
We can obtain it from $m(\lambda)$ by changing all $m_i(\lambda)$ into the vacuum diagrams centered at the $c_i$-shifted central lines and then reconstituting the total Maya diagram.
In Figure \ref{fig:core}, we show that for $r=3$ and $\lambda=(4,2,2)$ we obtain $\core(\lambda)= (1,1)$.

\begin{figure}[h]
\begin{tikzpicture}[scale=.5]
\draw (-6,.5) node {$m(\lambda)$:};;
\draw (-4,0) node {$\cdots$};;
\draw (-3,0) node {3};;
\draw (-3,1) circle (5pt);;
\draw (-2,0) node {2};;
\draw[fill=black] (-2,1) circle (5pt);;
\draw (-1,0) node {1};;
\draw[fill=black] (-1,1) circle (5pt);;
\draw (0,0) node {0};;
\draw (0,1) circle (5pt);;
\draw (.5,.5)--(.5,1.5);;
\draw (1,0) node {-1};;
\draw (1,1) circle (5pt);;
\draw (2,0) node {-2};;
\draw[fill=black] (2,1) circle (5pt);;
\draw (3,0) node {-3};;
\draw[fill=black] (3,1) circle (5pt);;
\draw (4,0) node {-4};;
\draw (4,1) circle (5pt);;
\draw (5,0) node {-5};;
\draw[fill=black] (5,1) circle (5pt);;
\draw (6,0) node {$\cdots$};;
\draw (0.5,-2) node {$\downarrow$};;
\draw (-3,-9) node {$\cdots$};;
\draw (-2,-9) node {2};;
\draw (-2,-8) circle (5pt);;
\draw (-1,-9) node {1};;
\draw[fill=black] (-1,-8) circle (5pt);;
\draw (0,-9) node {0};;
\draw (0,-8) circle (5pt);;
\draw (.5,-8.5)--(.5,-7.5);;
\draw (1,-9) node {-1};;
\draw (1,-8) circle (5pt);;
\draw (2,-9) node {-2};;
\draw[fill=black] (2,-8) circle (5pt);;
\draw (3,-9) node {-3};;
\draw[fill=black] (3,-8) circle (5pt);;
\draw (4,-9) node {-4};;
\draw[fill=black] (4,-8) circle (5pt);;
\draw (5,-9) node {-5};;
\draw[fill=black] (5,-8) circle (5pt);;
\draw (6,-9) node {$\cdots$};;
\draw (.5,-7)--(1.5, -6)--(-.5,-4)--(-1.5, -5)--(0.5,-7);;
\draw (-.5,-6)--(.5,-5);;
\end{tikzpicture}
    \caption{Obtaining $\core(\lambda)$ for $\lambda=(4,2,2)$ when $r=3$.}
    \label{fig:core}
\end{figure}
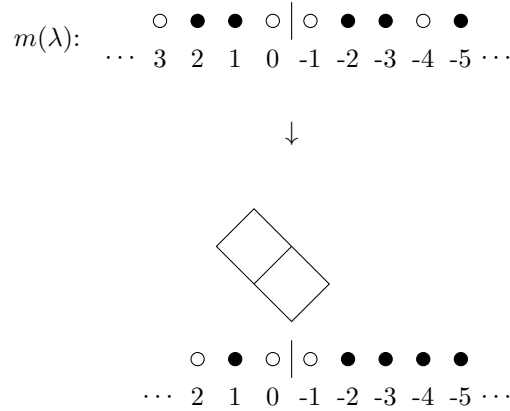

\vspace{.1in}

\noindent
We can obtain any $r$-core in this way, and so they are in bijection with tuples of charges $(c_0,\ldots, c_{r-1})$ such that \[
c_0+\cdots+c_{r-1}=0,
\]
due to $m(\lambda)$ being charge zero.
Thus, we naturally view $(c_0,\ldots, c_{r-1})$ as an element of the $A_{r-1}$ root lattice $Q$. 

\begin{prop}\label{CoreQuotDec}
We have the following:
\begin{enumerate}
\item The core-quotient decomposition gives a bijection
\[
\{\hbox{\rm partitions}\}\leftrightarrow\{\hbox{\rm$r$-core partitions}\}\times\left\{ \hbox{\rm$r$-multipartitions} \right\}
\]
\item The map $\lambda\mapsto (c_0,\ldots, c_{r-1})$ restricts to a bijection between $r$-core partitions and the $A_{r-1}$ root lattice $Q$.
\item The number of addable minus removable boxes of each color can be read from the charges:
\begin{equation}
\#A_i(\lambda)-\#R_i(\lambda)=\delta_{i,0}+c_{i-1}-c_i.
\label{ChargeAR}
\end{equation}
\item Adding an equal number boxes of each color to a partition does not change its core. In other words, if $\lambda \subset_k \mu$, then $\core(\lambda) = \core(\mu)$.
\end{enumerate}
\end{prop}

We can obtain new partitions by applying permutations to the quotient Maya diagrams.
\begin{defn}
Given $\pi \in \Sigma_r$ (permuting $\{0,1,\dots,r-1\}$) and a partition $\lambda$ with quotient Maya diagrams $\left( m_0(\lambda),m_1(\lambda),\ldots, m_{r-1}(\lambda) \right)$, we denote by $\pi\lambda$ the partition whose Maya diagrams are given by
\[
\left( m_{0}(\pi \lambda), m_{1} (\pi \lambda),\ldots, m_{r-1}( \pi \lambda) \right):=\left( m_{\pi(0)}(\lambda), m_{\pi(1)}(\lambda),\ldots, m_{\pi(r-1)}(\lambda) \right) = \pi m(\lambda).
\]
Letting $w_0$ be the longest element, so that $w_0(i) = (-i-1),$ we have
\[
m(w_0\lambda)\coloneqq \left( m_{r-1}(\lambda), \dots, m_1(\lambda), m_0(\lambda)    \right).
\]
Letting $\sigma = ( r-1 ~\cdots ~ 1~ 0)$ be the long cycle, $\sigma\lambda$ is the partition whose quotient Maya diagrams have been cyclically shifted:
\[
\left( m_0(\sigma\lambda),m_1(\sigma\lambda),\ldots, m_{r-1}(\sigma\lambda) \right)\coloneqq \left( m_{r-1}(\lambda), m_0(\lambda),\ldots, m_{r-2}(\sigma\lambda) \right).
\]
\end{defn}
The above actions send $r$-cores to $r$-cores, and the charges $(c_0,\ldots, c_{r-1})$ undergo the same permutation with respect to their coordinates.
We leave it to the reader to see that in our running example of $r=3$ and $\lambda=(4,2,2)$, we have 
\begin{align*}
w_0\lambda=(6,4), && 
\sigma\lambda= (6,4,1) && \text{ and } && 
\sigma^2\lambda= (5,3).
\end{align*}
In \cite{RW}, we show that
\begin{prop}\label{RevProp}
For any $\pi \in \Sigma_r$, the partitions $\lambda$ and $\pi\lambda$ have the same set of color $0$ boxes. Furthermore, if $\{\pi(0),\pi(1),\dots, \pi(k-1)\} = \{0,1,\dots,k-1\}$, then $\lambda$ and $\pi \lambda$ have the same color $k$ boxes.

In particular,
\begin{enumerate}
\item The partitions $\lambda$ and $w_0\lambda$ have the same set of color $0$ boxes.
\item For any $k\in\ZZ/r\ZZ$, $\lambda$ and $\sigma^k\lambda$ have the same set of color $0$ boxes.
\item For any $k\in\ZZ/r\ZZ$, $\lambda$ and $w_0\sigma^{-k}\lambda=\sigma^kw_0\lambda$ have the same set of color $k$ boxes.
\end{enumerate}
\end{prop}

\section{Symmetric functions}
\subsection{Classical symmetric functions}
A standard reference for some of this material is \cite{Mac}.
The ring of symmetric functions, $\Lambda$, is generated algebraically by any of the following families:
\begin{itemize}
    \item the power sums: $p_n = x_1^n + x_2 ^n + \cdots , $
    \item the complete homogeneous functions: $h_n = \sum_{i_1 \leq\cdots \leq i_n } x_{i_1} \cdots x_{i_n},$
    \item the elementary functions: $e_n = \sum_{i_1< \cdots < i_n} x_{i_1} \cdots x_{i_n}$. 
\end{itemize}
This means that for partitions $\lambda$, the elements $p_\lambda = p_{\lambda_1} \cdots p_{\lambda_{\ell(\lambda)}}$, $h_\lambda = h_{\lambda_1} \cdots h_{\lambda_{\ell(\lambda)}}$, and $e_\lambda = e_{\lambda_1} \cdots e_{\lambda_{\ell(\lambda)}}$, each give a basis for $\Lambda$. 
There is also the Schur function basis $\{s_\lambda\}$ and monomial basis $\{m_\lambda\}$. The symmetric function $m_\lambda$ is the sum of monomials whose exponents are given by $\lambda_1,\dots, \lambda_{\ell(\lambda)}$.

The Hall scalar product is defined by setting 
\[
\langle p_\lambda, p_\mu \rangle = z_\lambda \delta_{\lambda,\mu},
\]
where $z_\lambda = \prod_i i^{m_i} m_i!$, $m_i$ being the multiplicity of $i$ in $\lambda$. One also has that $\{s_\lambda\}$ forms an orthonormal basis; and the monomial basis is dual to the complete homogeneous basis. 

Given any expression $E(t_1,t_2,\dots)$ in the indeterminates $t_1,t_2,\dots$, one defines the plethystic substitution of a symmetric function $f$ at $E$, denoted $f[E]$, by first defining
\[
p_n[E(t_1,t_2,\dots)] = E(t_1^n,t_2^n\dots),
\]
then extending algebraically so that $p_\lambda[E] = p_{\lambda_1} [E] \cdots p_{\lambda_{\ell(\lambda)}}[E].$

The plethystic exponential
\[
\Omega[X] \coloneqq\exp \left( \sum_{k >0} p_k[X]/k \right) = \sum_{n \geq 0} h_n[X]
\]
produces the Cauchy kernel:
\[
\Omega[XY] = \sum_{\lambda} s_\lambda[X] s_\lambda[Y].
\]
It is the reproducing kernel under the Hall scalar product: for any $f\in \Lambda$,
\[
\left \langle \Omega[XY], f[X] \right \rangle = f[Y].
\]
Another useful property of the plethystic exponential is that for any expressions $X$ and $Y$,
\[
\Omega[X+Y] = \Omega[X]\Omega[Y].
\]

\subsection{Multisymmetric functions and new notations for plethystic expressions}
In this section, we will introduce new notation for vector (or matrix) plethysms that will help us write many of the operators that follow. 
First, recall that the ring of multisymmetric functions, $\Lambda^{\otimes r}$, is generated by
elements of the form $G = g_0[X^{(0)}]\cdots g_{r-1}[X^{(r-1)}],$ where each $g_i \in \Lambda$ is a symmetric function evaluated at a family of variables $X^{(i)} = x_{i1}+x_{i2}+\cdots$.
We now define vector plethysms in the following way:
\\

Let $\epsilon_0,\dots, \epsilon_{r-1}$ 
be a set of standard basis elements indexed by elements of $\ZZ/r\ZZ$ (so that $\epsilon_i = \epsilon_j$ if $i\equiv j \mod r$). 
The basis vectors will also behave like projections, so that \[\epsilon_i \epsilon_j = \epsilon_i \delta_{i,j}.\]
Given any expressions $Y^{(0)}, Y^{(1)},\dots, Y^{(r-1)}$, we define the vector $ Y^{\bullet} = \epsilon_0 Y^{(0)} + \cdots + \epsilon_{r-1}Y^{(r-1)}$. 
\begin{defn}
    Vector plethysms are defined by the following rules:
    \begin{itemize}
        \item Given $G = g_0[X^{(0)}]\cdots g_{r-1}[X^{(r-1)}] \in \Lambda^{\otimes r}$, we define
\[
G[Y^\bullet] \coloneqq g_0[\epsilon_0 Y^\bullet] \cdots g_{r-1}[ \epsilon_{r-1} Y^\bullet].
\]
        \item For an ordinary symmetric function $g \in \Lambda$ and non-vector expression $E$, 
        \[
g[ \epsilon_i E] = g[E],
        \]
        so that the $\epsilon_i$ can be treated as the constant $1$. In particular, combined with the first rule,
        \[
G[Y^\bullet]= g_0[  Y^{(0)}] \cdots g_{r-1}[ Y^{(r-1)}].
        \]
    \end{itemize}
\end{defn}

For any family of symmetric functions $\{F_\lambda\}$ indexed by partitions, we set
\[
F_{\vec{\lambda}}[X^\bullet] \coloneqq F_{\lambda^0}[X^{(0)}] \cdots F_{\lambda^{r-1}}[X^{(r-1)}].
\]
We then have the following families of bases for $\Lambda^{\otimes r}$:
\begin{align*}
    \{p_{\vec{\lambda}}[X^{\bullet}]\}, && 
    \{h_{\vec{\lambda}}[X^{\bullet}]\}, &&
    \{e_{\vec{\lambda}}[X^{\bullet}]\}, && \{m_{\vec{\lambda}}[X^{\bullet}]\},&&  \text{ and } &&
    \{s_{\vec{\lambda}}[X^{\bullet}]\}.
\end{align*}
The Hall scalar product extends to the multisymmetric setting. In particular, we have
\[
\langle s_{\vec{\lambda}}, s_{\vec{\mu}} \rangle = \delta_{\vec{\lambda},\vec{\mu}}.
\]

Lastly, it will be useful to denote the operator of multiplication by $F[X^\bullet] \in\Lambda^{\otimes r}$ by $\underline{F}[X^\bullet]$. The adjoint of this operator under the Hall scalar product will be denoted by $F^\perp[X^\bullet]$.

\subsubsection{Linear maps and plethystic operators}
We now consider linear transformations acting on the span of the $\epsilon_i$ coordinates. In particular, one defines
\begin{align*} \sigma \epsilon_i = \epsilon_{i-1} && \text{ and } && \iota \epsilon_i = \epsilon_{-i}. 
\end{align*}
Then for $G = g_0[X^{(0)}] \cdots g_{r-1}[X^{(r-1)}] \in \Lambda^{\otimes r}$,
\begin{align*}
G[ \sigma  X^{\bullet}] &= G[ \epsilon_{r-1}X^{(0)} + \epsilon_{0} X^{(1)} + \cdots + \epsilon_{r-2} X^{(r-1)}]
= g_0[X^{(1)}] \cdots g_{r-2}[X^{(r-1)}] g_{r-1}[X^{(0)}], \text{ and }
\\
G[\iota X^{\bullet}] &= 
G[ \epsilon_{0}X^{(0)} + \epsilon_{-1} X^{(1)} + \cdots + \epsilon_{-r+1} X^{(r-1)}]
= g_0[X^{(0)}]g_{1}[X^{(r-1)}] \cdots  g_{r-1}[X^{(1)}].
\end{align*}
One also has, for instance,
\[
p_n\left[ \epsilon_i (1-q \sigma) X^\bullet \right]
= p_n\left[ \epsilon_i X^\bullet - \epsilon_i q \sigma X^\bullet \right]
= p_n\left[ X^{(i)} -  q   X^{(i+1)} \right].
\]
\begin{defn}
For any $k$, let
\begin{align*}
M_k \coloneqq (1-q \sigma^k) (1-t \sigma^{-k}) && \text{ and }&&
M \coloneqq M_1 = (1-q\sigma) (1-t\sigma^{-1}),
\end{align*}
so that $M_0 = (1-q)(1-t)$ and $M_{-1} = (1-q \sigma^{-1})(1-t \sigma) = M^T$, where $M^T$ denotes the transpose. We then also have that $\iota M_k = M_{-k} \iota$.
\end{defn}

\begin{defn}
    For any transformation $A$, we define the translation operators $\TT^{(i)}_{A}$ and multiplication operators $\PP^{(i)}_{A}$ as follows:
    \begin{align*}
        \TT^{(i)}_{A} G[ X^{\bullet} ] &= G[  X^{\bullet} +  A\epsilon_i ],\\
        \PP^{(i)}_{A} G[ X^{\bullet} ] 
        & = \Omega\left[ \epsilon_i A   X^{\bullet} \right]G[ X^{\bullet} ] .
    \end{align*}
%Given a vector of expressions $Y^\bullet$, we also let
%\[
%\PP_{Y^\bullet} G[X^\bullet] \coloneqq \Omega\left[ Y^\bullet X^\bullet \right] G[X^{\bullet}]
%= \left(\prod_{i \in \ZZ/r\ZZ} \PP^{(i)}_{Y^\bullet} \right) G[X^\bullet].
%\]
\end{defn}
For instance, $\PP^{(i)}_1$ corresponds to multiplication by $\Omega[ X^{(i)}]$, and more crucially, $\PP^{(i)}_{M^T}$ corresponds to multiplication by
\begin{align*}
\Omega\left[ \epsilon_i  \frac{ X^{\bullet}}{(1-q \sigma^{-1}) (1-t \sigma)}\right] 
&  = 
\Omega\left[ \epsilon_i  
\sum_{a,b \geq 0} q^a t^b \sum_{ j \in \ZZ/r\ZZ} \epsilon_{j+a-b} X^{(j)}\right]   \\
& = \Omega\left[ 
\sum_{j \in \ZZ/r\ZZ} X^{(j)} \sum_{ \substack{a,b\geq 0 \\ j  \equiv i-a+b \mod r}} q^a t^b 
\right] 
\\
& = \Omega\left[ 
 \sum_{ a,b \geq 0} q^a t^b X^{(i-a+b)}
\right] .
\end{align*}

\begin{prop}
    For transformations $A$ and $B$, we have the following commutation relation:
    \[
\TT^{(i)}_A \PP^{(j)}_B = \Omega \left[ BA_{ji} \right] \PP^{(j)}_B \TT^{(i)}_A,
    \]
where  $BA_{ji} = \epsilon_j BA \epsilon_i\Big|_{\epsilon_j=1}$ is the matrix coefficient of $BA$ corresponding to the $\epsilon_j$ component in $BA \epsilon_i$. 
\end{prop}
\begin{proof}
We note that 
\begin{align*}
\TT^{(i)}_A \Omega\left[ \epsilon_j B X^{\bullet} \right] 
& = \Omega\left[ \epsilon_j B (X^{\bullet} + A\epsilon_i) \right] \TT^{(i)}_A \\
& = \Omega\left[ \epsilon_j B X^{\bullet} \right] \Omega \left[ \epsilon_j B A\epsilon_i \right] \TT^{(i)}_A\\
& = \Omega \left[ BA_{ji} \right]\Omega\left[ \epsilon_j B X^{\bullet} \right] \TT^{(i)}_A.\qedhere
\end{align*}
\end{proof}

\begin{prop}
    For $f \in \Lambda$ and any linear transformation $A$, we have
    \[
f\left[ X^\bullet A Y^\bullet \right] 
= f\left[ Y^\bullet A^T X^\bullet \right]
    \]
\end{prop}
\begin{proof}
For standard basis elements $\epsilon_i$ and $\epsilon_j$, we have
\begin{align*}
\epsilon_i X^{(i)} A \epsilon_j Y^{(j)} = X^{(i)} Y^{(j)}  A_{ij}\epsilon_i && \text{ and }&& \epsilon_j Y^{(j)} A^T \epsilon_i X^{(i)} = Y^{(j)} X^{(i)} A_{ij} \epsilon_j.
\end{align*}
Since $f \in \Lambda$, we will in the end treat the $\epsilon_i$ as $1$, meaning the equality in the proposition holds.
\end{proof}

\section{Wreath Macdonald polynomials and the Cauchy kernel}

\begin{defn}\label{ModDef}
For any partition $\lambda$, the \textit{modified wreath Macdonald polynomial} $H_\lambda=H_\lambda[X^\bullet;q,t]$ is characterized by the following conditions:
\begin{enumerate}
\item $H_\lambda[(1-q\sigma^{-1})X^\bullet]$ lies in the span of $\left\{ s_{\quot(\mu)}\, :\,~ \mu\ge_r\lambda  \right\}$;
\item $H_\lambda[(1-t^{-1}\sigma^{-1})X^\bullet]$ lies in the span of $\left\{ s_{\quot(\mu)}\, :\,~\mu\le_r\lambda  \right\}$;
\item when expanded in terms of $\{s_{\vec{\mu}}\}$, the coefficient of $s_n[X^{(0)}]$ is $1$, i.e. $H_\lambda[\epsilon_0 ]=1$.
\end{enumerate}
Here, $\lambda \leq_r \mu$ denotes dominance order, with the extra condition that $\core(\lambda) = \core(\mu)$. 
%and $|\quot(\lambda)| = |\quot(\mu)|$.
\end{defn}

\begin{defn}
The wreath Macdonald pairing is given by setting
\[
\langle F, G \rangle_\ast = \langle F[-\iota X^\bullet], G[M^T X^\bullet] \rangle.
\]
\end{defn}
The wreath Macdonald basis is orthogonal to the ``dagger'' basis,
\[
H_\lambda^\dagger\left[ X^\bullet\right] = H_\lambda [ - \iota X^\bullet; q^{-1}, t^{-1}],
\]
where one has
\[
\langle H_\lambda^\dagger , H_\mu \rangle_\ast = N_\lambda \delta_{\lambda,\mu}
\]
The value $N_\lambda$ is given below:

\begin{prop}[\protect{\cite[Theorem 3.32]{OSWreath}}]
   For $\square=(a,b)\in\lambda$, we define the following:
   \begin{itemize}
       \item \textit{Arm length}: $a_\lambda(\square)\coloneq\lambda_{b+1}-a$,
       \item \textit{Leg length}: $l_\lambda(\square)\coloneq {}^t\lambda_{a+1}-b$,
       \item \textit{Hook length}: $h_\lambda(\square)=a_\lambda(\square)+
       l_\lambda(\square)+1$.
   \end{itemize}
   We then have
   \[
   N_\lambda=\prod_{\substack{\square\in\lambda\\ h(\square)\equiv 0\hbox{ \footnotesize{mod} }r}}
   \left(1-q^{a_\lambda(\square)+1}t^{-l_\lambda(\square)}\right)
   \left(1-q^{-a_\lambda(\square)}t^{l_\lambda(\square)+1}\right).
   \]
\end{prop}

The \textit{Cauchy kernel} in the wreath case is given by
\begin{align*}
\Omega\left[- (\iota X^\bullet) \frac{Y^\bullet}{M^T}\right] &=\Omega\left[- X^\bullet \frac{(\iota Y^\bullet)}{M}\right] \\
&= \sum_{\core(\mu) = \alpha} 
\frac{H_\mu^\dagger[X^\bullet] H_\mu[Y^\bullet]}{N_\mu} 
= \sum_{\core(\mu) = \alpha} 
\frac{H_\mu[X^\bullet] H_\mu^\dagger[Y^\bullet]}{N_\mu},
\end{align*}
where $\alpha$ is any fixed core.
Note here that 
we are multiplying expressions involving the coordinates $\epsilon_i$. So, for instance,
\begin{align*}
\iota(X^\bullet) Y^\bullet &= (\epsilon_0 X^0 + \epsilon_{-1} X^{(1)} + \cdots + \epsilon_{1-r} X^{r-1})(\epsilon_0 Y^{(0)} + \cdots + \epsilon_{r-1} Y^{(r-1)})
\\
& = \epsilon_0 X^{(0)}Y^{(0)} + \epsilon_1 X^{(r-1)}Y^{(1)}+ \cdots + \epsilon_{r-1} X^{(1)} Y^{(r-1)}.
\end{align*}

These are some of the benefits of utilizing coordinates and vector plethysms as defined---we are able to treat expressions in the plethystic brackets as algebraic expressions.

\subsection{An extension of multisymmetric functions}
We let $\mathcal{W} = \Lambda^{\otimes r} \otimes \mathbb{C}[Q]$,
where $\mathbb{C}[Q]$ is the span of elements $e^{(c_0,\dots, c_{r-1})}$ with $c_0+\cdots +c_{r-1} = 0$ (that is, an element of the root lattice for $A_{r-1}$). Such an element $(c_0,\dots, c_{r-1})$ corresponds to an $r$-core, say $\alpha$, for which we instead write $e^{(c_0,\dots, c_{r-1})} = e^\alpha$. 

The star and Hall inner products on $\mathcal{W}$ are defined so that 
\begin{align*}
\left \langle 
F \otimes e^\alpha, G \otimes e^\beta
\right\rangle_\ast 
 \coloneqq \langle F , G \rangle_\ast \delta_{w_0\alpha,\beta}
 && \text{ and } && \left \langle 
F \otimes e^\alpha, G \otimes e^\beta
\right\rangle
 \coloneqq \langle F , G \rangle  \delta_{w_0\alpha,\beta}.
\end{align*}

\section{Operators on wreath Macdonald polynomials}
\begin{defn}
    We define the Delta operators by setting, for any $G \in \Lambda^{\otimes r}$, 
    \begin{align*}
G[D] \left( H_\lambda \otimes e^{\core(\lambda)} \right) = G[D^{\bullet}_\lambda] H_\lambda \otimes  e^{\core(\lambda)} \text{ and } \\
G[D^\dagger] \left(  H^\dagger_\lambda \otimes e^{w_0\core(\lambda)} \right) = G[D^{\bullet}_\lambda] H^\dagger_\lambda \otimes  e^{w_0 \core(\lambda)}.
    \end{align*}
\end{defn}
Note here that $D^{\bullet}_\lambda = \sum_{i \in \ZZ/r\ZZ} \epsilon_i D^{(i)}_\lambda$, so that
\[
G[\epsilon_i D] H_\lambda \otimes e^{\core(\alpha)} 
= G[ \epsilon_i D^{(i)}_\lambda] H_\lambda \otimes e^{\core(\alpha)} ,
\]
and likewise,
since
\begin{align*}
\epsilon_i \frac{ D_\lambda^\bullet}{M}
& =\epsilon_i \left( \sum_{a,b\geq0 } q^{a}t^{b} \sigma^{a-b} \sum_{ j \in \ZZ/r\ZZ} \epsilon_{j} D^{(j)}_\lambda \right)
\\
& =\epsilon_i \left( \sum_{a,b\geq0 }   \sum_{ j \in \ZZ/r\ZZ} \epsilon_{j-a+b} q^{a}t^{b} D^{(j)}_\lambda \right)
\\
& =\epsilon_i \left( \sum_{ p \in \mathbb{Z}/r\mathbb{Z}} \epsilon_p \sum_{a,b\geq0 }  q^{a}t^{b} D^{(p+a-b)}_\lambda \right) 
\\
& = \left( \frac{D_\lambda}{M_0} \right)^{(i)},
\end{align*}
we have that
\begin{equation}\label{epsilonDM}
G\left[\epsilon_i \frac{D}{M}\right] H_\lambda \otimes e^{\core(\alpha)} 
= G\left[ \epsilon_i \left(\frac{D_\lambda}{M_0}\right)^{(i)} \right] H_\lambda \otimes e^{\core(\alpha)} .
\end{equation}
One can also define a Delta operator given by 
\begin{align*}
G[B] \left( H_\lambda \otimes e^{\core(\lambda)} \right) = G[B^{\bullet}_\lambda] H_\lambda \otimes  e^{\core(\lambda)} \text{ and } \\
G[B^\dagger] \left(  H^\dagger_\lambda \otimes e^{w_0\core(\lambda)} \right) = G[B^{\bullet}_\lambda] H^\dagger_\lambda \otimes  e^{w_0 \core(\lambda)},
    \end{align*}
though we make less use of this version in our discourse.

We now define a family of nabla operators:
\begin{defn}
For any color $j$, let
\[T^{(j)}_\lambda \coloneqq \prod_{ \substack{ \square \in \lambda/\core(\lambda) \\ \bar{c}_\square = j}} (- \chi_\square) ~~ = ~(-1)^{|\quot(\lambda)|}\frac{e_{d_j(\lambda)}\left[B_\lambda^{(j)}\right]}{e_{d_j(\core(\lambda))}\left[ B_{\core(\lambda)}^{(j)} \right]}.\]
Then define 
\[
\nabla^{(j)} H_\lambda \otimes e^{\core(\lambda)}
\coloneqq
T^{(j)}_\lambda H_\lambda \otimes e^{\core(\lambda)}.
\]
We also set $\nabla \coloneqq\nabla^{(0)}$.
\end{defn}
It is shown in \cite{RW} that the following is true.
\begin{prop}
    For any partition $\lambda$,
    \[
T^{(0)}_\lambda = H_{w_0\lambda}[\iota D_{\core(\lambda)}^\bullet] = H_\lambda[ \iota D_{w_0 \core(\lambda)}^\bullet ] = T^{(0)}_{w_0 \lambda}.
    \]
\end{prop}

\section{The Tesler Identity and the map $\mathsf{V}$}
From \cite{RW}, we have a map 
\[\mathsf{V}
\coloneqq\nabla \PP^{(0)}_{-1/M^T} \TT^{(0)}_{1} \nabla 
\]
that relates wreath Macdonald polynomials to the exponential versions
\begin{align*}
\mathbb{E}_\lambda & \coloneqq 
 \Omega\left[- X^{\bullet} \frac{D^\bullet_\lambda}{M} \right] \\
& = \sum_{\core(\mu) = \alpha} \frac{ H^\dagger_\mu[X^\bullet] H_\mu[\iota D^\bullet_\lambda]}{N_\mu}
=\sum_{\core(\mu) = \alpha} \frac{ H_\mu[X^\bullet] H^\dagger_\mu[\iota D^\bullet_\lambda]}{N_\mu}.
\end{align*}
It has the following properties:
\begin{prop} \label{MainPropertiesofV}
    For any $F \in \Lambda^{\otimes r}$, we have
    \begin{itemize}
        \item $\mathsf{V} \underline{F} = F[\iota D^{\dagger}] \mathsf{V}$,
        \item $F^{\perp}[X^\bullet] \mathsf{V} = \mathsf{V}  F \left[ \frac{-D}{M} \right]$,
        \item $\mathsf{V}(1 \otimes e^{\alpha}) = \mathbb{E}_\alpha \otimes e^{\alpha},$
        \item $\mathsf{V} \nabla^{-1} H_{\lambda}[X^\bullet] \otimes e^{\core(\lambda)} = \mathbb{E}_\lambda \otimes e^{\core(\lambda)}.$
    \end{itemize}
    In particular, we have
    \begin{align*}
\TT^{(i)}_v \mathsf{V} = \mathsf{V} \Omega\left[ -v\epsilon_i
\frac{   D}{M}\right] && \text{ and }&& \mathsf{V} \PP_{-v/M^T}^{(i)} = \Omega\left[- v \epsilon_i\frac{ \iota D^\dagger}{M^T } \right] \mathsf{V}.
    \end{align*}
\end{prop}

\subsection{Wreath Macdonald-Koornwinder reciprocity revisited}
We now look at some consequences of the
 properties in Proposition \ref{MainPropertiesofV} that follow from our vector plethysm rules. We will treat this section as an ``exercise'' on using vector plethysms, while highlighting some of the important properties of wreath Macdonald polynomials. Let $\core(\lambda) = \alpha$.
 First note that from the last item in Proposition \ref{MainPropertiesofV},
\begin{align*}
\mathsf{V} H_\lambda \otimes e^{\alpha} 
 = T^{(0)}_\lambda \Omega\left[ - X^\bullet \frac{ D^\bullet_\lambda}{M} \right] \otimes e^{\alpha}
 = T^{(0)}_\lambda \sum_{w_0 \core(\mu) = \alpha} \frac{ H^\dagger_\mu[X^\bullet] H_\mu[\iota D^\bullet_\lambda]}{N_\mu} \otimes e^{\alpha}.
\end{align*}
On the other hand, one has
\begin{align*}
    \mathsf{V} H_\lambda \otimes e^{\alpha} 
& = H_\lambda[ \iota D^\dagger] \mathsf{V} ( 1 \otimes e^\alpha) \\
& = H_\lambda[ \iota D^\dagger] \sum_{w_0 \core(\mu) = \alpha} \frac{ H^\dagger_\mu[X^\bullet] H_\mu[\iota D^\bullet_\alpha]}{N_\mu} \otimes e^{\alpha} \\
& =  \sum_{w_0 \core(\mu) = \alpha} \frac{ H_\lambda[ \iota D^\bullet_\mu]H^\dagger_\mu[X^\bullet] H_\mu[\iota D^\bullet_\alpha]}{N_\mu} \otimes e^{\alpha}.
\end{align*}
Equating the coefficient of $H_\mu^\dagger[X^\bullet]$ in both expressions gives
\begin{prop}[Reciprocity]\label{Recip}
    For $\lambda$ and $\mu$ with $w_0 \core(\mu) = \core(\lambda)$, 
    \[
\frac{ H_\lambda[ \iota D_\mu^\bullet]}{T^{(0)}_\lambda} = \frac{H_\mu[\iota D_\lambda^\bullet]}{T^{(0)}_\mu}.
    \]
\end{prop}

We now extend Proposition \ref{Recip} to a wreath Macdonald-Koornwinder reciprocity by using our vector plethysm rules.
We first record degrees with an extra variable $z$, writing
\[
\nabla \PP^{(0)}_{-z/M^T} \TT^{(0)}_{1/z}  H_\lambda[z X^\bullet] 
= \Omega\left[ -z X^{\bullet} \frac{D^\bullet_\lambda}{M}\right] = \mathbb{E}_\lambda[zX^\bullet].
\]
We will now apply $\nabla^{-1} \mathsf{V} \nabla^{-1} =  \PP^{(0)}_{-1/M^T} \TT^{(0)}_{1}  $ to $\mathbb{E}_\lambda[zX^{\bullet}]$ two different ways: Let $\core(\lambda) = \alpha$, and consider the following computations in $\Lambda^{\otimes r} \otimes e^\alpha$ (so we can omit the $-\otimes e^{\alpha}$ in every line).
\begin{enumerate}
    \item First, we write
    \begin{align*}
        \PP^{(0)}_{-1/M^T} \TT^{(0)}_{1}  \mathbb{E}_\lambda[zX^{\bullet}]
        & = \Omega\left[ - \epsilon_0 \frac{X^{\bullet}}{M^T} \right] \mathbb{E}_\lambda[z(X^{\bullet} +\epsilon_0) ]\\
        &=
        \Omega\left[ - \epsilon_0 \frac{X^{\bullet}}{M^T} \right] \Omega\left[ -z (X^{\bullet} + \epsilon_0) \frac{ D^\bullet_\lambda}{M}\right] \\
        &=
        \Omega\left[ - X^{\bullet} \frac{\epsilon_0}{M} \right] \Omega\left[ - X^{\bullet} \frac{z D^\bullet_\lambda}{M}\right]
        \Omega\left[ -z \epsilon_0 \frac{ D^\bullet_\lambda}{M}\right] \\
       & =
        \Omega\left[ - X^{\bullet} \frac{ \epsilon_0 + z D^\bullet_\lambda}{M} \right]
        \Omega\left[ -z \left(  \frac{ D_\lambda}{M_0}\right)^{(0)}\right] \\
        & = \Omega\left[ -z \left(  \frac{ D_\lambda}{M_0}\right)^{(0)}\right]
        \sum_{w_0\core(\mu) = \alpha} 
        \frac{H_\mu^\dagger[X^\bullet] H_\mu[\epsilon_0 + z \iota D_\lambda^\bullet]}{N_\mu}.
    \end{align*}
\item On the other hand, 
\begin{align*}
    \nabla^{-1} \mathsf{V} \nabla^{-1} \mathbb{E}_\lambda[zX^{\bullet}] 
    & = 
     \nabla^{-1} \mathsf{V} \nabla^{-1} 
     \nabla
     \PP_{-z/M^T}^{(0)} \TT^{(0)}_{1/z} 
     H_\lambda[zX^{\bullet}]
     \\
      & = 
     \nabla^{-1} \mathsf{V} 
     \Omega\left[ - z \epsilon_0 \frac{X^{\bullet}}{M^T} \right] 
     H_\lambda[ \epsilon_0 + zX^{\bullet}]
    \\
     & = 
     \nabla^{-1} \Omega\left[ -z \epsilon_0 \frac{\iota D^{\dagger}}{M^T}  \right] 
     H_\lambda[\iota(\epsilon_0+ z \iota D^{\dagger}) ]
      \mathsf{V} \left(1\right)\\
      & = 
       \Omega\left[ -z \epsilon_0  \frac{D^{\dagger}}{M} \right] 
    H_\lambda[\epsilon_0 + z \iota D^{\dagger} ]
      \nabla^{-1} \mathbb{E}_\alpha[X^\bullet]
     \\
      & = 
      \sum_{ w_0\core(\mu) = \alpha}
      \frac{H_{ \mu}^\dagger[X^\bullet]}{N_\mu} 
      \Omega\left[ -z \left(\frac{D_{\mu}}{{M_0}}\right)^{(0)} \right] 
     H_\lambda[\epsilon_0 + z \iota D_{\mu}^{\bullet} ].
\end{align*}
\end{enumerate}
Comparing the coefficient of $H_\mu^\dagger[X^\bullet]$ in both of our computations, and using that 
\[
\Omega\left[ -z \left( \frac{D_\mu}{M_0}\right)^{(0)}\right] = \Omega\left[ -z  B_\mu^{(0)} \right] \Omega\left[ - \left(\frac{z}{M_0}\right)^{(0)} \right] = \left(\prod_{ \substack{\square \in \mu \\ \bar{c}_\square = 0}} 1-z\chi_\square\right) \Omega\left[ - \left(\frac{z}{M_0}\right)^{(0)} \right],
\]
we get that
\begin{prop}[Wreath Macdonald-Koornwinder Reciprocity]\label{wreathMKReciprocity}
If $w_0 \core(\mu) = \core(\lambda)$, then
\[
\frac{H_\lambda[ \epsilon_0 + z \iota D_\mu^\bullet]}{ \prod_{ \substack{\square \in \lambda \\ \bar{c}_\square = 0}} 1-z\chi_\square}
=
\frac{H_\mu[ \epsilon_0 + z \iota D_\lambda^\bullet]}{ \prod_{ \substack{\square \in \mu \\ \bar{c}_\square = 0}} 1-z\chi_\square}.
\]
More generally, it is shown in \cite{RW} that if $ w_0 \sigma^{-k}\core(\mu) = \core(\lambda)$, then 
\begin{equation}
 \frac{H_\lambda[ \epsilon_0 + z  \iota \sigma^k D_\mu^\bullet]}{ \prod_{ \substack{\square \in \lambda \\ \bar{c}_\square = k}} 1-z\chi_\square}
=
\frac{H_\mu[ \epsilon_0 + z  \iota \sigma^k D_\lambda^\bullet]}{ \prod_{ \substack{\square \in \mu \\ \bar{c}_\square = k}} 1-z\chi_\square}.   
\label{ShiftMK}
\end{equation}
\end{prop}

\begin{rem}
    The plethysms in \eqref{ShiftMK} look different from the ones from \cite{RW}.
    This is due to the different treatment of plethysm in this paper.
    The two formalisms coincide in their action on multi-symmetric functions.
    On the other hand, in \textit{loc. cit.}, we do not define $\sigma$ and $\iota$ on a sum of characters like $D_\mu^\bullet$; rather, an expression such as $f[\sigma^k\iota D_\mu^\bullet]$ is meant to be interpreted as performing $f[\sigma^k\iota X^\bullet]$ and then evaluating at $D_\mu^\bullet$.
    This results in the sequence of transformations
    \[
    p_n[X^{(i)}]\mapsto p_n[X^{(k-i)}]\mapsto p_n[D_\mu^{(k-i)}].
    \]
    In the formalism used in this paper, we have
    \[
    \epsilon_i\left(\iota\sigma^k D_\mu^\bullet\right)=D_\mu^{(k-i)}.
    \]
    Therefore, the evaluations in \cite{RW} and \eqref{ShiftMK} coincide.
\end{rem}

For a given diagram $\lambda$, let
\[
\Pi_{\lambda}^{(k)} = \prod_{\substack{\square \in \lambda/(1) \\ \bar{c}_\square = k}} 1- \chi_\square.
\]
\begin{cor} 
    If $w_0 \sigma^{-k}\core(\mu) =  \core(\lambda)$, then
    \[
\frac{H_\lambda[ \epsilon_0 - \epsilon_{k} + \iota \sigma^k  M^T B_\mu^{\bullet}]}{\Pi_\lambda^{(k)}}
=
\frac{H_\mu[ \epsilon_0 - \epsilon_{k} +\iota   \sigma^k  M^T B_\lambda^{\bullet}]}{\Pi_\mu^{(k)}}
.
    \]
In particular, for $k= 0$, we have
\[
\frac{H_\lambda[ M \iota B_\mu^\bullet]}{\Pi_{\lambda}^{(0)}}
=
\frac{H_\mu[ M \iota B_\lambda^\bullet]}{\Pi_{\mu}^{(0)}}.
\]
\end{cor}
\subsection{For empty cores}
If $\lambda$ has an empty core, then Theorem \ref{wreathMKReciprocity}
gives some interesting identities that are more visibly similar to the identities for ordinary modified Macdonald polynomials.
\begin{prop}
    If $\core(\lambda)= \varnothing$, then
    \[
H_\lambda\left[ \epsilon_0 - z \epsilon_{k} \right] = \prod_{ \substack{ \square \in \lambda \\ \bar{c}_\square = k}} 1- z \chi_\square.
    \]
\end{prop}

\begin{cor}\label{h_e_inner_product}
    For any color $k$ and for $\lambda \vdash n $ with $\core(\lambda) = \varnothing$, 
\[
\langle H_\lambda, h_{n-j}[X^{(0)}] e_j[X^{(k)}] \rangle
= e_j\left[ B^{(k)}_\lambda \right].
\]
This means that for any homogeneous $F \in \Lambda^{\otimes r}$ of degree $n$ and any $r$-core $\alpha$, 
\begin{align*}
\left \langle e_j[\epsilon_k B ]~ F \otimes e^\alpha , h_n[X^{(0)}] \otimes e^{\alpha}  \right \rangle
= \left \langle F  , h_{n-j}[X^{(0)}]e_j[ X^{(k)} ]  \right \rangle.
\end{align*}
\end{cor}
\begin{proof}
    We have that 
    \begin{align*}
    \prod_{ \substack{ \square \in \lambda \\ \bar{c}_\square = k}} 1- z \chi_\square
    & = 
H_\lambda[ \epsilon_0 - z \epsilon_{k}]
\\& = \sum_{i+j= n} h_i[X^{(0)}]^{\perp}  (-z)^j e_{j}[X^{(k)}]^{\perp} H_{\lambda} 
\\
& = \left\langle H_\lambda, \sum_{j=0}^n h_{n-j}[X^{(0)}] (-z)^j e_{j}[X^{(k)}] \right\rangle .
    \end{align*} 
    Then take the coefficient of $(-z)^j$ on both sides.
\end{proof}

When $j = n$ and $k = 0$ in Corollary \ref{h_e_inner_product}, we get the following evaluation:
\begin{cor}\label{cor:emptycore_evaluation} For $\lambda$ with $\core(\lambda) = \varnothing$, 
\[
H_\lambda[-\epsilon_0] = T^{(0)}_\lambda.
\]
\end{cor}
It is generally unclear at this point how to compute $H_\lambda[-\epsilon_0]$ for general $\lambda$. 
However, one of the consequences of the formulas for Pieri coefficients is giving a tableau formula for this evaluation.
In \cite{OSWreath}, the authors present an unpublished result of Haiman that describes $H_\lambda[-\epsilon_i]$.

\section{Five-term relations and tableau formulas for Pieri coefficients}\label{5TermSec}
To lighten the notation, we will let 
\begin{align*}
\Delta_v^{(s)} =  \Omega\left[-v \epsilon_s \frac{D}{M}  \right],
&& \text{with inverse} && \Delta_{-v}^{(s)} =  \Omega\left[v \epsilon_s \frac{D}{M} \right].
\end{align*}
We see from \eqref{epsilonDM} that 
\[
\Delta_v^{(s)} H_\lambda = \Omega\left[ - v \left( \frac{D_\lambda}{M_0} \right)^{(s)} \right] H_\lambda.
\]
We wish to study the action of the operator
\[W(u,v) = \sum_{i,j} u^i v^j W_{i,j} = \PP^{(p)}_{u/{M^T}}
\Delta_v^{(s)} \PP^{(p)}_{-u/{M^T}} \Delta_{-v}^{(s)}.
\]

Such expressions for modified Macdonald polynomials were studied by Garsia and Mellit \cite{Garsia-Mellit}, leading to a five-term relation with incredibly powerful consequences. We will also find that for the wreath Macdonald polynomials there is a family of ``five-term relations,'' depending on two colors and whether we act on the wreath Macdonald polynomials or their dagger version.
This will lead to explicit formulas for Pieri and dual Pieri coefficients.
However, since wreath Macdonald polynomials are not self-dual, we will need some care in approaching the dual Pieri case.

In this section, we first look at Pieri formulas corresponding to multiplication by a modified elementary symmetric function. In the subsequent section, we will look at the dual Pieri formulas corresponding to skewing with complete homogeneous symmetric functions.
\\

One of the main ingredients we require is the following condition on Pieri coefficients. 
\begin{prop}[\protect{\cite[Lemma 4.3]{RW}}] 
Recall that we write $\lambda \subset_k \mu$ if $\mu$ is attained from $\lambda$ by adding $k$ cells of each color, and we therefore also have that
$\core(\lambda) = \core(\mu)$. 
For any $k$ and partition $\lambda$,
\[
{e}_k\left[ \epsilon_p \frac{X^{\bullet}}{M^T}\right] H_\lambda = \sum_{\lambda \subset_k \mu} d_{\mu,\lambda}^{(p)} H_\mu.
\]
 Similarly, the adjoint of this is given by 
\[
h_k^{\perp}\left[ X^{(p)}\right] H_\lambda^\dagger = \sum_{\mu \subset_k \lambda} c_{\mu,\lambda}^{\dagger (p)} H_\mu^\dagger.
\]
\end{prop}
Our goal in this section is to give a formula for these coefficients.
\subsection{Triangularities in $u$ and $v$}
The main goal is to show that $W(u,v) = \sum_{i} u^{i} v^i W_{i,i},$ then identify the diagonal terms. We do this in a similar way to the proof in \cite{Garsia-Mellit}: we first show that $W_{i,j} = 0$ for $j>i$, identify $W_{i,i}$, then show $W_{i,j} = 0$ for $i>j$.
\begin{prop}
The operator
\[ \Delta_v^{(s)} \underline{e}_k\left[ \epsilon_p \frac{X^{\bullet}}{M^T}\right] \Delta_{-v}^{(s)}   \]
is a polynomial in $v$ of degree $k$. Furthermore, the top degree in $v$ is given by
\begin{align*} \Delta_v^{(s)} \underline{e}_k\left[ \epsilon_p \frac{X^{\bullet}}{M^T}\right] 
\Delta_{-v}^{(s)}
\Big|_{ v^k} 
= \nabla^{(s)} \underline{e}_k\left[ \epsilon_p \frac{X^{\bullet}}{M^T}\right] {\nabla^{(s)}}^{-1}.
\end{align*}
In particular, since $\PP^{(p)}_{u/M^T}$ only raises the degree in $u$, this means that $W_{i,j} = 0$ for $j>i$; and for $i=j$, we have
\[W_{i,i} = \nabla^{(s)} e_{i}\left[ \epsilon_p \frac{X^{\bullet}}{M^T}\right] {\nabla^{(s)}}^{-1}.\]
\end{prop}
\begin{proof} 
From the Pieri coefficients, we see that when $\core(\lambda) = \alpha$,
\begin{align*}
   \Delta_v^{(s)}
\underline{e}_k\left[ \epsilon_p \frac{X^{\bullet}}{M^T}\right] & \Delta_{-v}^{(s)}
\left( H_\lambda \otimes e^{\alpha} \right) \\
&
=  \Delta_v^{(s)} \underline{e}_k\left[ \epsilon_p \frac{X^{\bullet}}{M^T}\right]
\Omega\left[v \epsilon_s \frac{D_\lambda^{\bullet}}{M} \right]
\left(H_\lambda \otimes e^{\alpha} \right) \\
& =  \Omega \left[v \left(\frac{D_\lambda}{M_0}\right)^{(s)} \right] \Delta_v^{(s)}
 \left(\sum_{\core(\mu) = \alpha} d_{\mu,\lambda}^{(p)} H_\mu \otimes e^{\alpha} \right) \\
 & = \sum_{\core(\mu) = \alpha} d_{\mu,\lambda}^{(p)}  \Omega\left[v \left(\frac{D_\lambda}{M_0}\right)^{(s)} \right]  \Omega\left[-v \left( \frac{D_\mu}{M_0}\right)^{(s)}\right]  
H_\mu \otimes e^{\alpha}  \\
& = \sum_{\core(\mu) = \alpha} d_{\mu,\lambda}^{(p)}  \Omega\left[-v \left( \frac{D_\mu - D_\lambda}{M_0} \right)^{(s)}\right] 
H_\mu \otimes e^{\alpha} \\
& = 
\sum_{\core(\mu) = \alpha} d_{\mu,\lambda}^{(p)} 
\left(\prod_{{ \substack{ {\square \in \mu/ \lambda} \\ {\bar{c}(\square) = s} }}} 1-v\chi_{\square}  \right)H_\mu \otimes e^{\alpha}.
\end{align*}

Since $\mu$ is attained from $\lambda$ by adding $k$ cells of each color, this is a polynomial in $v$ of degree $k$.
Note that we get $W$ by then applying $\PP^{(p)}_{u/M^T}$ on the left, which only raises the power of $u$. Therefore, the power of $v$ is always at most the degree of $u$, meaning $W_{i,j} = 0$ for $j>i$.

For the other statement, we note that for the coefficient of $v^k$ we have 
\begin{align*}
\sum_{\core(\mu) = \alpha} d_{\mu,\lambda}^{(p)} 
\left(\prod_{{ \substack{ {\square \in \mu/ \lambda} \\ {\bar{c}(\square) = s} }}} 1-v\chi_{\square}\right) H_\mu \otimes e^{\alpha} \Big|_{v^k} 
& = 
\sum_{\core(\mu) = \alpha} d_{\mu,\lambda}^{(p)} \left(
\prod_{{ \substack{ {\square \in \mu/ \lambda} \\ {\bar{c}(\square) = s} }} } -\chi_{\square}\right) H_\mu \otimes e^{\alpha}  \\ 
& = 
\sum_{\core(\mu) = \alpha} d_{\mu,\lambda}^{(p)}
\left(T_\mu^{(s)}/T_{\lambda}^{(s)} \right)
 H_\mu \otimes e^{\alpha}  \\
&=
\nabla^{(s)} \underline{e}_k\left[ \epsilon_p \frac{X^{\bullet}}{M^T}\right] {\nabla^{(s)}}^{-1}
H_\lambda \otimes e^{\alpha} .\qedhere
\end{align*}
\end{proof}

\begin{prop}
The operator
\[
\PP^{(p)}_{u/M^T}
 \underline{e}_k \left[ \epsilon_s  \frac{D}{M} \right] \PP^{(p)}_{-u/M^T}
\]
is a polynomial in $u$ of degree $k$. In particular, this means that 
$W_{i,j} = 0$ for $j<i$. 
\end{prop}
\begin{proof}
We see that 
\begin{align*}
\mathsf{V}^{-1} \mathsf{V}\PP^{(p)}_{u/M^T}
\Delta_v^{(s)} \PP^{(p)}_{-u/M^T}
& = 
\mathsf{V}^{-1} \Omega \left[ u \epsilon_p \frac{\iota D^{\dagger }}{M^T} \right]
\TT_{v}^{(s)} \Omega \left[ - u \epsilon_p \frac{\iota D^{\dagger }}{M} \right]
\mathsf{V}.
\end{align*}
From the Pieri rules on the basis $H_\mu^{\dagger} \otimes e^{\alpha}$ (with $\core(w_0\mu)= \alpha$), we once again see that 
\[
\Omega \left[ u \epsilon_p \frac{\iota D^\dagger}{M^T} \right]~
h_k^{\perp}\left[X^{(s)}\right] \Omega \left[ - u \epsilon_p \frac{\iota D^\dagger}{M^T}\right]\]
is a polynomial in $u$ of degree $k$. Therefore, the degree of $u$ in
$
\PP^{(p)}_{u/M^T}
\Delta_v^{(s)}\PP^{(p)}_{-u/M^T}
$
is at most the degree of $v$. Since the operator $ \Omega[v \epsilon_s ( D /M)]$ only increases the degree in $u$, we have that
$W_{i,j} = 0$ for $j<i$.
\end{proof}

\begin{thm} \label{thm:five-term-relation}
We have the following five-term relation for any two colors $p$ and $s$:
\[\nabla^{(s)} \PP_{-uv/M^T}^{(p)} {\nabla^{(s)}}^{-1} = 
\PP^{(p)}_{u/M^T}
\Delta_v^{(s)} \PP^{(p)}_{-u/M^T} \Delta_{-v}^{(s)}.
\]
\end{thm}
\begin{proof}
We combine the previous two propositions. 
\end{proof}

\subsection{Rules for Pieri Coefficients}
We rewrite Theorem \ref{thm:five-term-relation} as follows:
\[
\PP^{(p)}_{-u/M^T} \nabla^{(s)} \PP_{-v/M^T}^{(p)} {\nabla^{(s)}}^{-1} =
\Delta_{v/u}^{(s)}  \PP^{(p)}_{-u/M^T} 
\Delta_{-v/u}^{(s)}.
\]
Taking the coefficient of $u^k v^l$ on the left-hand side yields
\begin{equation}
\PP^{(p)}_{-u/M^T} \nabla^{(s)} \PP_{-v/M^T}^{(p)} {\nabla^{(s)}}^{-1} \Big|_{u^k v^l} 
= (-1)^{k+l} 
\underline{e}_k\left[ \epsilon_p \frac{X^\bullet}{M^T} \right] \nabla^{(s)} \underline{e}_l\left[ \epsilon_p \frac{X^\bullet}{M^T} \right] {\nabla^{(s)}}^{-1}.
\label{eq:k-perp-l-perp}
\end{equation}
Taking the same coefficient on the right-hand side gives
\begin{align}
\begin{split}
\Delta_{v/u}^{(s)}  \PP^{(p)}_{-u/M^T} &\Delta_{-v/u}^{(s)}\Big|_{u^k v^l} \\
&=
\sum_{i = 0 }^k  (-1)^i e_i\left[ \epsilon_s  \frac{D}{M} \right]
(-1)^{k+l} \underline{e}_{k+l}\left[ \epsilon_p \frac{X^{\bullet}}{M^T} \right] 
 h_{k-i}\left[ \epsilon_s  \frac{D}{M} \right].
\label{eq:k+l-perp}
\end{split}
\end{align}

We now apply these operators to $(H_\lambda \otimes e^{\alpha})$, where $\core(\lambda) = \alpha$. 
Applying the operator in Equation \ref{eq:k-perp-l-perp} and extracting the coefficient of $(H_\mu \otimes e^{\alpha})$ gives
\[ (-1)^{k+l}
\sum_{ \lambda\subset_l \nu \subset_k \mu}
d_{\mu,\nu}^{(p)}  d_{\nu,\lambda}^{(p)} \frac{T^{(s)}_\nu}{{T^{(s)}_\lambda}}.
\]
Applying instead Equation \ref{eq:k+l-perp} and again extracting the coefficient of $(H_\mu \otimes e^{\alpha})$, we get
\[
(-1)^{k+l} d_{\mu,\lambda}^{(p)} \sum_{i = 0 }^k (-1)^i e_i \left[  \left(\frac{D_\mu}{M}  \right)^{(s)} \right]
 h_{k-i}\left[ \left(\frac{D_\lambda}{M_0}\right)^{(s)} \right]
 = (-1)^{k+l} d_{\mu,\lambda}^{(p)} h_k \left[ \left( \frac{D_\lambda - D_\mu}{M }\right)^{(s)}\right].
\]
As a consequence, we find that 
\begin{thm}\label{kPieri}
    For any positive integers $k$ and $l$, and any partition $\lambda \subset_{k+l} \mu$,
\begin{equation}
    d_{\mu,\lambda}^{(p)} = \frac{(-1)^k}{e_k\left[ B_{\mu/\lambda}^{(s)}\right]}
    \sum_{\lambda \subset_l \nu \subset_k \mu}  d_{\mu, \nu}^{(p)} d_{\nu,\lambda}^{(p)} T_{\nu/\lambda}^{(s)}.
\end{equation}
This gives a recursion for $d_{\mu,\lambda}^{(p)}$, with the base case where $\lambda\subset_1\mu$ is given by Proposition \ref{Deg1DLem}. 
\end{thm}

\begin{rem}
 In particular, we obtain a recursion for $d_{\mu,\lambda}^{(p)}$ for each choice of color $s$, and the resulting coefficient is independent of this choice.
One can even choose different values of $s$ at each step of the recursion.   
\end{rem}

Theorem \ref{kPieri} is incredibly powerful, as it allows one to compute Pieri coefficients in terms of the degree $1$ Pieri rules in a very quick way.
The following is computed in the appendix:

\begin{prop}\label{Deg1DLem}
For $\lambda\subset_1 \mu$, let $\square_i\in\mu/\lambda$ be the box of color $i$ and $\chi_i=\chi_{\square_i}$.
We then have
\begin{equation}
d_{\mu,\lambda}^{(p)}
=\frac{1}{qt-1}\left(\frac{\chi_0}{\chi_p}\right)
\prod_{i\in\ZZ/r\ZZ}
\frac{\displaystyle\prod_{\square\in R_i(\mu)}\left(1-q^{-1}t^{-1}\frac{\chi_i}{\chi_\square}\right)}
{\displaystyle\prod_{\square\in A_i(\mu)}\left(1-\frac{\chi_i}{\chi_\square}\right)}.
\label{Deg1D}
\end{equation}

\end{prop}

We should note that there is also an adjoint version of our relations, which can be written in the following way:
\begin{equation} \label{eq:dagger-five-term}
 {\nabla^{\dagger (s)}}^{-1} 
\TT^{(p)}_{-uv} 
\nabla^{\dagger (s)}
=
\Delta_{-v }^{\dagger (s)}  \TT_u^{(p)}
\Delta_{v }^{\dagger (s)}
\TT^{(p)}_{-u} .
\end{equation}
However, these give dual Pieri formulas for the dual 
basis $\{H_\lambda^{\dagger} \}$. One can write recursions similarly to how we do for the dual Pieri formulas in Theorem \ref{thm:dual-pieri-formula}, leading to formulas for the monomial expansion of the dagger basis.
\\

We now give formulas for the dual Pieri rules for wreath Macdonald polynomials.

\subsection{The dual Pieri formulas}
For $\lambda \subset_k \mu$, let $c_{\lambda,\mu}^{(p)}$ be the Pieri coefficients appearing in the expansion
\[
h_{k}^\perp[X^{(p)}] H_\mu = \sum_{\lambda \subset_k \mu} c_{\lambda,\mu}^{(p)} ~ H_\lambda.
\]
Our goal in this section is to give a formula for these coefficients. 
\\

Define the ``downarrow'' operator by setting
\[
\downarrow F[X^\bullet;q,t] \otimes e^{\alpha} = F[-\iota X^\bullet;q^{-1},t^{-1}] \otimes e^{w_0 \alpha}.
\]

On operators, we have the following:
\begin{align*}
\downarrow \nabla \downarrow = \nabla^{-1} ,
&&
\downarrow \TT^{(i)}_{A} \downarrow = \TT^{(i)}_{-\iota A_\ast}  ,
&&
\downarrow \PP^{(i)}_A \downarrow = \PP^{(-i)}_{-A_\ast \iota},
&& \downarrow F[\epsilon_s D^\dagger] \downarrow
= F[\epsilon_s D_\ast].
\end{align*}
Conjugating $\mathsf{V}$ by this transformation, we get a new operator:
\[
\mathsf{V}_\ast = \nabla^{-1} \PP_{1/M_\ast}^{(0)} \TT_{-1}^{(0)} \nabla^{-1} .
\]
We now let \[
{}^\ast\mathbb{E}_\lambda = \Omega \left[ 
\frac{\iota X^\bullet}{M} qt (D^\bullet_\lambda)_\ast
\right]
=
\Omega \left[ 
 X^\bullet \frac{ qt \iota (D^\bullet_\lambda)_\ast}{M}
\right].
\]
Note that
\[
\TT_v^{(i)}\Omega \left[ 
 X^\bullet \frac{ qt \iota (D^\bullet_\lambda)_\ast}{M}
\right]
= \Omega \left[ 
 X^\bullet \frac{ qt \iota (D^\bullet_\lambda)_\ast}{M}
\right]
\Omega \left[ 
\epsilon_i v \frac{ qt \iota (D^\bullet_\lambda)_\ast}{M}
\right].
\]
We also define
\begin{align*}
{}^\ast\Delta_{v}^{(i)} \coloneqq \Omega\left[ -v \epsilon_i \frac{ D_\ast}{M_\ast}  \right] &&
\text{ and its inverse}
&&
{}^\ast\Delta_{-v}^{(i)} \coloneqq \Omega\left[ v \epsilon_i \frac{ D_\ast}{M_\ast}  \right].
\end{align*}
In particular, note that since
\begin{align*}
\epsilon_i \frac{(D_\lambda^\bullet)_\ast}{M_\ast} 
& = \epsilon_i \sum_{a,b \geq 0 } q^{-a} t^{-b} \sigma^{a -b } \sum_{ j \in \mathbb{Z}/r\mathbb{Z}} \epsilon_j (D^{(j)}_\lambda)_\ast \\
& = \epsilon_i \sum_{a,b \geq 0 } q^{-a} t^{-b}  \sum_{ j \in \mathbb{Z}/r\mathbb{Z}} \epsilon_{j-a+b} (D^{(j)}_\lambda)_\ast \\
& = \epsilon_i \sum_{ p \in \mathbb{Z} / r\mathbb{Z}} \epsilon_p \left( \sum_{a,b \geq 0 } q^{a} t^{b}  D^{(p+a-b)}_\lambda \right)_\ast \\
& = \epsilon_i \left( \frac{D_\lambda}{M_0} \right)^{(i)}_\ast,
\end{align*}
we have
\[
{}^\ast \Delta_v^{(i)} H_\lambda
= \Omega\left[-v \left( \frac{D_\lambda}{M_0} \right)^{(i)}_\ast  \right] H_\lambda.
\]

\begin{prop} We have the following properties of $\mathsf{V^\ast}$:
\begin{itemize}
    \item $\mathsf{V}_\ast \underline{F} = F[-D_\ast] \mathsf{V}_\ast$,
    \item $F^\perp[X^\bullet] \mathsf{V}_\ast = \mathsf{V}_\ast F\left[ \epsilon_i  \frac{qt \iota D^{\dagger}_\ast}{M}  \right]$ and, in particular, $\TT^{(i)}_{\pm v} \mathsf{V}_\ast = \mathsf{V}_\ast \Omega\left[ \pm v\epsilon_i  \frac{qt \iota D^{\dagger}_\ast}{M}  \right]$,
    \item $\mathsf{V}_\ast (1 \otimes e^{w_0 \alpha} ) = {}^\ast \mathbb{E}_\alpha \otimes e^{w_0 \alpha}$ for any $r$-core $\alpha$,
    \item $\mathsf{V}_\ast (H_\lambda^\dagger \otimes e^{w_0 \core(\lambda)} ) = {}^\ast \mathbb{E}_\lambda \otimes e^{w_0 \core(\lambda)}$.
\end{itemize}
\end{prop}

\begin{thm}\label{thm:dual-five-term-relation}
For any colors $s$ and $p$,
    \[
{\nabla^{(s)}}\TT^{(p)}_{uv} 
{\nabla^{(s)}}^{-1} 
= \TT^{(p)}_{-u} {}^\ast\Delta_{-v }^{(s)}  \TT_{u}^{(p)}
{}^\ast \Delta_{v }^{(s)}.
    \]
\end{thm}
\begin{proof}
One can prove this in a similar way to the proof of Theorem \ref{thm:five-term-relation}.
    We have two important equalities:
    \begin{align*}
{}^\ast\Delta_{-v }^{(s)} h_k^{\perp}[X^{(p)}]
{}^\ast \Delta_{v }^{(s)}
H_\mu \otimes e^{\alpha}
& = \sum_{ \lambda \subset_k \mu} c_{\lambda,\mu}^{(p)} H_\lambda \Omega\left[v \epsilon_s \frac{(D_\lambda^\bullet - D_\mu^\bullet)_\ast}{M_\ast} \right]\\
&= \sum_{ \lambda \subset_k \mu} c_{\lambda,\mu}^{(p)} H_\lambda \prod_{ \substack{ \square \in \mu/\lambda \\ \bar{c}_\square = s}} (1-v \chi_\square^{-1}),
    \end{align*}
and 
    \begin{align*}
         \TT_{-u}^{(p)}
{}^\ast \Delta_{-v }^{(s)}
\TT^{(p)}_{u} \mathsf{V}_\ast
= \mathsf{V}_\ast \Omega\left[- u \epsilon_p  \frac{qt \iota D^\dagger_\ast}{M} \right]
\PP_{v/M_\ast}^{(s)} \Omega\left[ u \epsilon_p  \frac{qt \iota D^\dagger_\ast}{M} \right].
    \end{align*}
The first equality establishes one triangularity in $u$ and $v$, and the equality between the the diagonal term and the left-hand side of the proposition. The second equality establishes the second triangularity in $u$ and $v$.

Alternatively, one can conjugate Equation \eqref{eq:dagger-five-term} by $\downarrow$.
\end{proof}

First rewrite the the five-term relation as
\[
\TT^{(p)}_{u}
{\nabla^{(s)}}\TT^{(p)}_{v} 
{\nabla^{(s)}}^{-1} 
= {}^\ast\Delta_{-v/u }^{(s)}  \TT_{u}^{(p)}
{}^\ast \Delta_{v/u }^{(s)}.
\]
Taking the coefficient of $u^l v^k$ on the left-hand side applied to $H_\mu \otimes e^{\alpha}$, $\alpha = \core(\mu)$,
\begin{align*}
\TT^{(p)}_{u} 
    {\nabla^{(s)}} \TT^{(p)}_{v} 
{\nabla^{(s)}}^{-1}  H_\mu \otimes e^{\alpha} \Big|_{u^l v^k}
= 
\sum_{\lambda \subset_l \nu \subset_k \mu} c_{\mu,\nu}^{(p)} c_{\nu,\lambda}^{(p)} \frac{T^{(s)}_\nu}{T^{(s)}_\mu} H_\lambda.
\end{align*}
On the right-hand side, we have
\begin{align*}
     {}^\ast\Delta_{-v/u }^{(s)}  \TT_{u}^{(p)}
{}^\ast \Delta_{v/u }^{(s)} H_\mu \otimes e^{\alpha} \Big|_{u^l v^k}
= 
\sum_{\lambda \subset_{k+l} \mu} 
c_{\mu,\lambda}^{(p)} \Omega\left[ -v \epsilon_i \left( \frac{  D_\mu - D_\lambda}{M_0}  \right)^{(i)}_\ast \right] \Big|_{v^{k}} H_\lambda \otimes e^{\alpha}.
\end{align*}
Equating the coefficient of $H_\lambda \otimes \alpha$ in both expressions gives that
\begin{thm}\label{thm:star-dual-pieri-formula} Let $k$ and $l$ be positive integers. For any colors $p$ and $s$, and for $\lambda \subset_{k+l} \mu$,
    \begin{align*}
     c_{\lambda,\mu}^{(p)} =  \frac{(-1)^k}{e_k \left[\left(B^{(s)}_{\mu/\lambda}\right)_\ast \right]}  \sum_{\lambda \subset_l \nu \subset_k \mu}
     \frac{1}{T_{\mu/\nu}^{(s)}}
     c_{\nu,\mu}^{(p)} c_{\lambda,\nu}^{(p)} .
    \end{align*}
When $\lambda \subset_1 \mu$, we use the formula at the end of this section in Proposition \ref{Deg1CLem}.
\end{thm}
Note that the coefficients $M_{\lambda,\vec{\gamma}}(q,t) = M_{\lambda,\vec{\gamma}}$ in the monomial expansion
\[
H_\lambda[X^\bullet]  = \sum_{\vec{\gamma} \vdash |\quot(\lambda)|}
M_{\lambda,\vec{\gamma}}~ ~m_{\vec{\gamma}}[X^\bullet]
\]
can now be readily computed, since
\[
M_{\lambda,\vec{\gamma}} = 
\prod_{i=0}^{r-1} 
\prod_{j = 1}^{\gamma^i}
h_{\gamma^i_j}^{\perp}[X^{(i)}]
~~ H_\lambda.
\]
Therefore, Theorem \ref{thm:star-dual-pieri-formula} gives a tableau formula for the monomial expansion of wreath Macdonald polynomials.

\subsection{Without the starred Delta operators}
In this section, we will see how to get the dual five-term relation without the star on the Delta operators.
We start by writing Theorem \ref{thm:dual-five-term-relation} as
\[
\TT_{u}^{(p)} \nabla^{(s)} \TT_{uv}^{(p)} {\nabla^{(s)}}^{-1} 
= {}^{\ast}\Delta_{-v}^{(s)} \TT_u^{(p)} {}^\ast \Delta_{v}^{(s)}.
\]
Acting on $H_\lambda$ in $\Lambda^{\otimes r} \otimes e^{\core(\lambda)}$, 
\begin{align*}
{}^{\ast}\Delta_{-v}^{(s)} \TT_u^{(p)} {}^\ast \Delta_{v}^{(s)} H_\lambda 
& = 
{}^{\ast}\Delta_{-v}^{(s)} \TT_u^{(p)} \Omega\left[ -v \left( \frac{D_\lambda}{M_0}\right)_\ast^{(s)}\right] H_\lambda 
\\
& =
{}^{\ast}\Delta_{-v}^{(s)} \sum_{ k \geq 0 } u^k \sum_{\mu \subset_k \lambda} c_{\lambda,\mu}^{(p)} \Omega\left[ -v \left( \frac{D_\lambda}{M_0}\right)_\ast^{(s)}\right] H_\mu 
\\
& =
\sum_{ k \geq 0 } u^k \sum_{\mu \subset_k \lambda} c_{\lambda,\mu}^{(p)} \Omega\left[v\left(B_\mu - B_\lambda\right)_\ast^{(s)}\right] H_\mu 
\\& =
\sum_{ k \geq 0 } u^k \sum_{\mu \subset_k \lambda} c_{\lambda,\mu}^{(p)} \left(\prod_{ \substack{ \square \in \mu \\ \bar{c}_\square = s}} \frac{1}{1-v \chi_\square^{-1}}  \right)
\left(\prod_{ \substack{ \square \in \lambda \\ \bar{c}_\square = s}}  {1-v \chi_\square^{-1}}  \right)
H_\mu 
\\
& =
\sum_{ k \geq 0 } u^k \sum_{\mu \subset_k \lambda} c_{\lambda,\mu}^{(p)} \left(\prod_{ \substack{ \square \in \mu \\ \bar{c}_\square = s}} \frac{1}{1-v^{-1} \chi_\square }  \right)
\left(\prod_{ \substack{ \square \in \lambda \\ \bar{c}_\square = s}}  {1-v^{-1} \chi_\square }  \right)
v^{-d_s(\mu) + d_s(\lambda)} \frac{T_\mu^{(s)}}{T_\lambda^{(s)}}
H_\mu 
\\
& =
\sum_{ k \geq 0 } (uv)^k \sum_{\mu \subset_k \lambda} c_{\lambda,\mu}^{(p)} \left(\prod_{ \substack{ \square \in \mu \\ \bar{c}_\square = s}} \frac{1}{1-v^{-1} \chi_\square }  \right)
\left(\prod_{ \substack{ \square \in \lambda \\ \bar{c}_\square = s}}  {1-v^{-1} \chi_\square }  \right)
\frac{T_\mu^{(s)}}{T_\lambda^{(s)}}
H_\mu 
\\
& =   \nabla^{(s)} \Delta_{-1/v}^{(s)}   \TT^{(p)}_{uv} \Delta_{1/v}^{(s)} {\nabla^{(s)}}^{-1} H_\lambda.
\end{align*}
We then have
\begin{align*}
\TT_{u}^{(p)} \nabla^{(s)} \TT_{uv}^{(p)} {\nabla^{(s)}}^{-1} 
= \nabla^{(s)} \Delta_{-1/v}^{(s)}   \TT^{(p)}_{u v} \Delta_{1/v}^{(s)} {\nabla^{(s)}}^{-1},
\end{align*}
or rather,
\begin{align*}
{\nabla^{(s)}}^{-1}\TT_{u}^{(p)} \nabla^{(s)} 
=  \Delta_{-1/v}^{(s)}   \TT^{(p)}_{uv} \Delta_{1/v}^{(s)} \TT_{-uv}^{(p)} .
\end{align*}
Setting $v \rightarrow 1/v$, then setting $u \rightarrow uv$, we get the following version for the five-term relation, which, besides the colors, looks like the original version of Garsia and Mellit.
\begin{thm} For any colors $p$ and $s$,
 \[
{\nabla^{(s)}}^{-1} \TT^{(p)}_{uv} 
{\nabla^{(s)}} 
=  \Delta_{-v }^{(s)}  \TT_{u}^{(p)}
 \Delta_{v }^{(s)} \TT_{-u}^{(p)} .
    \]
    \end{thm}
By taking coefficients as we did in proving Theorem \ref{thm:star-dual-pieri-formula}, we get the following formula for the dual Pieri coefficients.
    \begin{thm}\label{thm:dual-pieri-formula} Let $k$ and $l$ be positive integers. For any colors $p$ and $s$, and for $\lambda \subset_{k+l} \mu$,
    \begin{align*}
     c_{\lambda,\mu}^{(p)} =  \frac{(-1)^k}{e_k \left[ B^{(s)}_{\mu/\lambda}  \right]}  \sum_{\lambda \subset_k \nu \subset_l \mu} 
     T_{\nu/\lambda}^{(s)}
     c_{\lambda,\nu}^{(p)} c_{\nu,\mu}^{(p)}  .
    \end{align*}
    In particular, this formula is independent of the choice of $s$. For the base case, when $\lambda \subset_1 \mu$, we have the formula in Proposition \ref{Deg1CLem}.
\end{thm}

\begin{prop}\label{Deg1CLem}
For $\lambda\subset_1\mu$, let $\square_i$ be the box in  $\mu/\lambda$ of color $i$, and let $\chi_i=\chi_{\square_i}$.
We then have
\begin{align}
c_{\lambda,\mu}^{(p)}
&= \frac{1}{1-qt}
\left(\frac{\chi_p}{\chi_0}\right)
\prod_{i\in\ZZ/r\ZZ} 
\frac{\displaystyle\prod_{\square\in A_i(\lambda)}  \left(1-qt\frac{\chi_i}{\chi_\square}\right)}
{\displaystyle\prod_{\square\in R_i(\lambda)} \left( 1-\frac{\chi_i}{\chi_\square} \right)}.
\label{1PieriDual}
\end{align}
\end{prop}
We will prove \eqref{1PieriDual} in the appendix.

\subsection{Another evaluation formula}
In this section, we will be concerned with the evaluation $H_\lambda[-\epsilon_0]$. As mentioned after Corollary \ref{cor:emptycore_evaluation}, it is not clear from reciprocity how one can attain this value when $\core(\lambda)$ is not empty. However, we can get a tableau formula using the Pieri coefficients.

Let ${c'}_{\lambda,\mu}^{(p)}$ denote the dual Pieri coefficients one gets from
\[
e_k^{\perp}[X^{(p)}] H_\mu
= \sum_{\lambda \subset_k \mu} 
{c'}_{\lambda,\mu}^{(p)} H_\lambda.
\]
Then, for any $\mu$,
\[
H_\mu[-\epsilon_0] = (-1)^n e_n^\perp[X^{(0)}] H_\mu
 = (-1)^n {c'}_{\core(\mu),\mu}^{(0)}.
\]

The inverse to the dual five-term relation can be written as
\[
{\nabla^{(s)}}^{-1} \TT^{(p)}_{-uv} 
{\nabla^{(s)}} 
= \TT_{u}^{(p)} \Delta_{-v }^{(s)}  \TT_{-u}^{(p)}
 \Delta_{v }^{(s)}  .
\]
Rewriting this as
\[
\TT_{-u}^{(p)} {\nabla^{(s)}}^{-1} \TT^{(p)}_{-v} 
{\nabla^{(s)}} 
=  \Delta_{-v/u }^{(s)}  \TT_{-u}^{(p)}
 \Delta_{v/u }^{(s)}
\]
and applying the left-hand side to $H_\mu \otimes e^{\alpha}$ with $\alpha = \core(\mu)$ yields 
\begin{align*}
    \TT_{-u}^{(p)} {\nabla^{(s)}}^{-1} \TT^{(p)}_{-v} 
{\nabla^{(s)}}  H_\mu \otimes e^{\core(\mu)} \Big|_{u^k v^l}
&
= (-1)^{k+l}
\sum_{\lambda \subset_k \nu \subset_l \mu}
{c'}_{\lambda,\nu}^{(p)} {c'}_{\nu,\mu}^{(p)} \frac{T_\mu^{(s)}}{T_\nu^{(s)}}
H_\lambda \otimes e^{\alpha}.
\end{align*}
The right-hand side, instead, gives
\begin{align*}
    \Delta_{-v/u }^{(s)}  \TT_{-u}^{(p)}
 \Delta_{v/u }^{(s)} H_\mu \otimes e^{\alpha}
 & =
 (-1)^{k+l}
 \sum_{\lambda \subset_{k+l} \mu} {c'}_{\lambda,\mu}^{(p)} 
 \Omega\left[ -v( B_\mu - B_\lambda)^{(s)} \right]\Big|_{v^l}
 H_\lambda \otimes e^\alpha
 \\
 & = (-1)^k e_l \left[ B_{\mu/\lambda}^{(s)} \right]
  \sum_{\lambda \subset_{k+l} \mu} {c'}_{\lambda,\mu}^{(p)} 
 H_\lambda \otimes e^\alpha.
\end{align*}
Equating the coefficients of $H_\lambda \otimes e^\alpha$, we get the following recursive formula.
\begin{thm}
    For any positive integers $k,l$, and $\lambda\subset_{k+l} \mu$, 
\begin{align*}
    {c'}_{\lambda,\mu}^{(p)}
= \frac{(-1)^l}{e_l \left[ B_{\mu/\lambda}^{(s)} \right]}
\sum_{\lambda \subset_k  \nu \subset_l \mu}
{c'}_{\lambda,\nu}^{(p)} {c'}_{\nu,\mu}^{(p)} T_{\mu/\nu}^{(s)}.
\end{align*}
The base case, when $\lambda \subset_1 \mu$, is computed via Proposition \ref{Deg1CLem}.
\end{thm}
As a consequence, one can get a formula for $H_\mu[-\epsilon_0] = (-1)^n {c'}_{\core(\mu),\mu}^{(0)}$.

\section{The $\mathbb{D},\mathbb{D}^\ast$-operators and wreath Theta operators}
In \cite{RW}, we define the following series of operators:
\begin{align*}
\DD & \coloneq
(1-qt)^{r-1}
 \Omega\left[ \sum_{i \in \ZZ/r\ZZ} q \frac{z_{i+1}}{z_i} + t \frac{z_i}{z_{i+1}} \right]
\prod_{i\in\ZZ/r\ZZ} \PP^{(i)}_{-z_i}
\prod_{i\in\ZZ/r\ZZ} z_{i}^{-C_{i}}\TT^{(i)}_{z_i^{-1}M^T},\\
\mathbb{D}^\ast&\coloneqq(1-q^{-1}t^{-1})^{r-1}
\Omega\left[\sum_{i \in \ZZ/r\ZZ} t^{-1} \frac{z_{i+1}}{z_i} + q^{-1} \frac{z_i}{z_{i+1}} \right]
\prod_{i \in \ZZ/r\ZZ} \PP_{z_i}^{(i)}
\prod_{i \in \ZZ/r \ZZ} z_i^{C_i} \TT_{-M_\ast/z_i}^{(i)}.
\end{align*}
where $M_\ast = (1-q^{-1}\sigma )(1-t^{-1}\sigma^{-1})$ and 
\[z_0^{-C_0}\cdots z_{r-1}^{-C_{r-1}} e^\alpha = z_0^{ c_0 - c_{r-1}} z_1^{ c_1 - c_0 } \cdots z_{r-1}^{c_{r-1} - c_{r-2} }~~~ \text{with $c_i - c_{i-1} = \#R_i(\alpha) - \#A_i(\alpha) +\delta_{i,0}$}.
\]
For $\mathbb{D}$, we expand all rational functions into Laurent series assuming
\begin{equation*}
    \begin{aligned}
        |z_i|&=1,&
        |q|,|t|&<1.
    \end{aligned}
\end{equation*}
On the other hand, for $\mathbb{D}^\ast$, we expand the rational functions into Laurent series assuming
\begin{equation*}
    \begin{aligned}
        |z_i|&=1,&
        |q|,|t|&>1.
    \end{aligned}
\end{equation*}
For a given vector $\vec{k} = (k_0,\dots, k_{r-1})$, one can then define
\begin{align*}
\DD_{\vec{k}} f \otimes e^{\alpha} = z_0^{k_0} \cdots z_{r-1}^{k_{r-1}} \DD f  \otimes e^{\alpha} \Big|_{z^0} &&\text{ and } &&
\DD_{\vec{k}}^\ast f \otimes e^{\alpha} = z_0^{k_0} \cdots z_{r-1}^{k_{r-1}} \DD^\ast f  \otimes e^{\alpha} \Big|_{z^0},
\end{align*}
where $|_{z^0}$ denotes taking the constant term with respect to the $z_i$ variables.
We will also write $\vec{k}$ as $\epsilon_0 k_0+\cdots +\epsilon_{r-1} k_{r-1}$.

The following are some properties of these operators.
\begin{prop}[\cite{RW}] \label{DoperatorProperties}
Let $\tilde{e}_1^{(i)}=e_1[ \epsilon_i (X^\bullet/M^T)],$ and let $\hat{e}_1^{(i)}=e_1[ \epsilon_i (X^\bullet/M)].$
\begin{enumerate}
\item We have $\DD_{\vec{0}} = \hat{e}_1^{(0)}[-D]$. Additionally,
\begin{align*}
  \mathbb{D}_{\epsilon_p - \epsilon_0} = \hat{e}_1^{(p)}\left[ -D \right]   && \text{ and } && 
  \mathbb{D}_{\epsilon_0 - \epsilon_p}^\ast = (\hat{e}_1^{(p)}\left[ -D \right])_\ast.
\end{align*}
\item For any vector $\vec{k}=(k_0,\ldots, k_{r-1})\in \ZZ^r$, the following hold:
\begin{align*}
\left[\mathbb{D}_{\vec{k}}, \tilde{e}_1^{(i)}\right] &=\mathbb{D}_{\vec{k}-\epsilon_{i}}, &
\left[\mathbb{D}_{\vec{k}}, e_1^\perp[X^{(i)}]\right]&= \mathbb{D}_{\vec{k}+\epsilon_{i}},&\\
\left[\mathbb{D}_{\vec{k}}^*, \tilde{e}_1^{(i)}\right]&= -qt\mathbb{D}_{\vec{k}-\epsilon_{i}}^*,&
\left[\mathbb{D}_{\vec{k}}^\ast, e_1^\perp[X^{(i)}]\right]&= -\mathbb{D}_{\vec{k}+\epsilon_{i}}^\ast.&
\end{align*}
\item We also have the following relations:
\begin{align*}
\nabla \tilde{e}_1^{(i)}\nabla^{-1}&= \mathbb{D}_{-\epsilon_i} &
\nabla^{-1} e_1^\perp[X^{(i)}]\nabla&=-\mathbb{D}_{\epsilon_i}, & \\
\nabla^{-1} \tilde{e}_1^{(i)}\nabla&= -qt\mathbb{D}^*_{-\epsilon_i},&
\nabla e_1^\perp[X^{(i)}]\nabla^{-1}&=\mathbb{D}_{\epsilon_i}^*.&
\end{align*}
\item We have
\begin{align*}
\mathbb{D}_{\vec{k}}^\dagger = \mathbb{D}_{-\iota\vec{k}} && \text{ and } &&
\left(\mathbb{D}_{\vec{k}}^*\right)^\dagger = (qt)^{-|\vec{k}|}\mathbb{D}_{-\iota\vec{k}}^*
\end{align*}
In particular, 
$\mathbb{D}_{\epsilon_{0}-\epsilon_{p}} = \hat{e}_1^{(-p)}[-D^\dagger]$ and $\mathbb{D}_{\epsilon_{p}-\epsilon_{0}}^* = (\hat{e}_1^{(-p)}[-D^\dagger])_\ast$.
\end{enumerate}
\end{prop}

There are more identities for different colored versions of $\nabla$. For instance, 
\begin{prop} \label{prop:nabla_e_D}
    For colors $s$ and $i$, 
    \[
\nabla^{(s)} \underline{\tilde{e}}_1^{(i)} {\nabla^{(s)}}^{-1} 
= \mathbb{D}_{\epsilon_{s} - \epsilon_0 - \epsilon_i }.
    \]
\end{prop}
\begin{proof}
One sees that for $\core(\lambda) = \alpha$,
\begin{align*}
    \nabla^{(s)} \tilde{e}_1^{(i)} {\nabla^{(s)}}^{-1} H_\lambda \otimes e^{\alpha}
    & =  \nabla^{(s)} \tilde{e}_1^{(i)}   H_\lambda/T^{(s)}_\lambda \otimes e^{\alpha} \\
    & = \nabla^{(s)} \sum_{\lambda \subset_1 \mu} 
     d_{\mu,\lambda}^{(i)} H_\mu/T^{(s)}_\lambda \otimes e^{\alpha} \\
    & = \sum_{\lambda \subset_1 \mu} 
    d_{\mu,\lambda}^{(i)} H_\mu 
    \left( T_\mu^{(s)} / T^{(s)}_\lambda \right) \otimes e^{\alpha} \\
    & = \sum_{\lambda \subset_1 \mu} 
    d_{\mu,\lambda}^{(i)} H_\mu 
    \left( - \left( D_\mu/M_0 \right)^{(s)} + \left( D_\lambda/M_0 \right)^{(s)} \right)  \otimes e^{\alpha} \\
     & =  \left(\mathbb{D}_{\epsilon_{s} - \epsilon_0} \tilde{e}_1^{(i)} - 
\tilde{e}_1^{(i)}\mathbb{D}_{\epsilon_{s} - \epsilon_0}\right) H_\lambda \otimes e^{\alpha} \\
 & = \left [ \mathbb{D}_{\epsilon_{s} - \epsilon_0},  \tilde{e}_1^{(i)}    \right ] H_\lambda \otimes e^{\alpha} \\
 & = \mathbb{D}_{\epsilon_{s} - \epsilon_0 - \epsilon_i } H_\lambda \otimes e^{\alpha}.\qedhere
\end{align*}
\end{proof}
We will also make use of the following fact, seen in Lemma 4.4 in \cite{RW}.
\begin{lem}\label{lemma:PD}
    For a color $p$,
    \[ \PP^{(p)}_{w/M^T} \mathbb{D}= \left(1-\frac{w}{z_p}\right)\mathbb{D}\PP^{(p)}_{w/M^T}.
    \]
\end{lem}
\subsection{Wreath Theta operators}
We have looked at the operator
\[
\Theta^{(p,s)}(w;v) \coloneqq\Delta_v^{(s)} \PP^{(p)}_{-w/M^T} \Delta_{-v}^{(s)}
\]
in the context of the five-term relations. We define an anaologue of the Theta operator, introduced in \cite{Theta}, as follows. For any $f \in \Lambda^{\otimes r}$ and color $(s)$, we let 
\[
\Theta^{(s)}_f = \Delta_v^{(s)} \underline{f} \Delta_{-v}^{(s)} \Big|_{v =1}.
\]
With this notation, we have
\[
\Theta^{(p,s)}(w;1) = \Theta^{(s)}_{\Omega[-w \epsilon_p (X^{\bullet}/M^T)]} = 
\sum_{ n \geq 0} (-w)^n \Theta^{(s)}_{\tilde{e}_n^{(p)} },
\]
where
\[
\tilde{e}_n^{(p)} = e_n \left[  \epsilon_p \frac{X^{\bullet}}{ M^T}\right].
\]
is a modified elementary symmetric function.

The following relation between Theta operators and the $\mathbb{D}$-operators is an analogue of Conjecture 10.3 in \cite{Theta}, which was proved in \cite{RomeroTheta}.
\begin{thm}
For given colors $p$, $s$, and $i$,
\[
\left[ \Theta^{(s)}_{\tilde{e}^{(p)}_n} , \mathbb{D}_{\epsilon_s-\epsilon_0 - \epsilon_i} \right]
= 
\sum_{k = 1}^{n} (-1)^{k} \mathbb{D}_{\epsilon_s-\epsilon_0-\epsilon_i - k \epsilon_p} \Theta^{(s)}_{\tilde{e}^{(p)}_{n-k}}.
\]
In particular, when $s= 0$, one has
\[
\left[ \Theta^{(0)}_{\tilde{e}^{(p)}_n} , \mathbb{D}_{-\epsilon_i } \right]
= 
\sum_{k = 1}^{n} (-1)^{k} \mathbb{D}_{-\epsilon_i - k \epsilon_p} \Theta^{(0)}_{\tilde{e}^{(p)}_{n-k}}.
\]
\end{thm}
\begin{proof}
By applying the five-term relations,  Proposition \ref{prop:nabla_e_D}, and Lemma \ref{lemma:PD},
we get
\begin{align*}
\Theta^{(p,s)}(w;v) 
\mathbb{D}_{\epsilon_s-\epsilon_0 - \epsilon_i} \left( \Theta^{(p,s)}(w;v)\right)^{-1}
& = 
\left(
\Delta_v^{(s)} \PP^{(p)}_{-{w}/{M^T}} \Delta_{-v}^{(s)} \right)
\nabla^{(s)} \tilde{e}_1^{(i)} 
{\nabla^{(s)}}^{-1}
\left(
\Delta_v^{(s)} \PP^{(p)}_{w/M^T} \Delta_{-v}^{(s)} \right)\\
& = \left(
\PP_{-w/M^T}^{(p)}
\nabla^{(s)} \PP_{-wv/M^T}^{(p)} {\nabla^{(s)}}^{-1} \right)
\nabla^{(s)} \tilde{e}_1^{(i)} 
{\nabla^{(s)}}^{-1}\\& \hspace{15em}
\left(\nabla^{(s)} \PP_{wv/M^T}^{(p)} {\nabla^{(s)}}^{-1}\PP_{w/M^T}^{(p)} \right)\\
& = 
\PP_{-w/M^T}^{(p)}
\nabla^{(s)} \PP_{-wv/M^T}^{(p)}   \tilde{e}_1^{(i)} 
 \PP_{wv/M^T}^{(p)} {\nabla^{(s)}}^{-1}\PP_{w/M^T}^{(p)}
 \\
 & = 
\PP_{-w/M^T}^{(p)}
\left(
\nabla^{(s)}    \tilde{e}_1^{(i)} 
{\nabla^{(s)}}^{-1} \right)\PP_{w/M^T}^{(p)}
\\
& = 
\PP_{-w/M^T}^{(p)}
\mathbb{D}_{\epsilon_s-\epsilon_0-\epsilon_i}
\PP_{w/M^T}^{(p)}
\\
& = 
\PP_{-w/M^T}^{(p)}
(z_s/z_0z_i) \mathbb{D}
\PP_{w/M^T}^{(p)} \Big|_{z^0}
\\
& = 
 (1-w/z_{p})^{-1} (z_s/z_0z_i) \mathbb{D} \Big|_{z^0} \\
& = 
\sum_{k \geq 0} w^k \mathbb{D}_{\epsilon_s-\epsilon_0 - \epsilon_i - k \epsilon_p}.
\end{align*}
Therefore,
\[
\Theta^{(p,s)}(w;v) 
\mathbb{D}_{\epsilon_s-\epsilon_0-\epsilon_i} = \sum_{k \geq 0} w^k \mathbb{D}_{\epsilon_s-\epsilon_0-\epsilon_i - k \epsilon_p} \Theta^{(p,s)}(w;v) 
\]
Taking the coefficient of $w^n$ on both sides gives
\[
\Theta^{(s)}_{\tilde{e}^{(p)}_n} \mathbb{D}_{\epsilon_s-\epsilon_0-\epsilon_i}
= \sum_{k=0}^n (-1)^k \mathbb{D}_{\epsilon_s-\epsilon_0-\epsilon_i - k \epsilon_p} \Theta^{(s)}_{\tilde{e}^{(p)}_{n-k}}.\qedhere
\]
\end{proof}
This expression allows one to write Theta operators in terms of $\mathbb{D}$-operators.

\appendix
\section{Degree 1 Pieri rules}
In this appendix, we prove the following formulas for the Pieri coefficients $d_{\mu,\lambda}^{(p)}$ and $c_{\lambda,\mu}^{(p)}$ from Section \ref{5TermSec}:
\begin{thm}
    For $\lambda \subset_1 \mu$, let $\square_i$ be the box in $\mu/\lambda$ of color $i$ and $\chi_i=\chi_{\square_i}$.
We then have
\begin{align}
\label{E1}
d_{\mu,\lambda}^{(p)}
&=\frac{1}{qt-1}\left(\frac{\chi_0}{\chi_p}\right)
\prod_{i\in\ZZ/r\ZZ}
\frac{\displaystyle\prod_{\square\in R_i(\mu)}\left(1-q^{-1}t^{-1}\frac{\chi_i}{\chi_\square}\right)}
{\displaystyle\prod_{\square\in A_i(\mu)}\left(1-\frac{\chi_i}{\chi_\square}\right)},\\
%&= \frac{1}{1-qt}\chi_p^{-1}
%\prod_{i\in\ZZ/r\ZZ} \chi_i^{-C_i}
%\frac{\displaystyle\prod_{\square\in R_i(\mu)}\left( 1-qt\frac{\chi_\square}{\chi_i} \right)}
%{\displaystyle\prod_{\square\in A_i(\mu)}\left(1-\frac{\chi_\square}{\chi_i}\right)}, \text{ and }\\
\label{H1}
c_{\lambda,\mu}^{(p)}
&= \frac{1}{1-qt}
\left(\frac{\chi_p}{\chi_0}\right)
\prod_{i\in\ZZ/r\ZZ} 
\frac{\displaystyle\prod_{\square\in A_i(\lambda)}  \left(1-qt\frac{\chi_i}{\chi_\square}\right)}
{\displaystyle\prod_{\square\in R_i(\lambda)} \left( 1-\frac{\chi_i}{\chi_\square} \right)}.
\end{align}
\end{thm}

\subsection{From Pieri to Delta}
Let $\alpha=\core(\lambda)$.
To compute \eqref{E1}, we will pass things through $\mathsf{V}$ and apply Proposition \ref{MainPropertiesofV}:
\begin{align*}
    \mathsf{V}\left(e_1\left[\epsilon_p\frac{X^\bullet}{M^T}\right]H_\lambda\otimes e^{\alpha}\right)
    &=\mathsf{V}\left(\sum_{\lambda\subset_1\mu}d_{\mu,\lambda}^{(p)}H_\mu\otimes e^{\alpha}\right)\\
    &=\sum_{\lambda\subset_1\mu}T_\mu^{(0)}d_{\mu,\lambda}^{(p)}\mathbb{E}_\mu\otimes e^{\alpha}.
\end{align*}
On the other hand, we have
\begin{align*}
    \mathsf{V}\left(e_1\left[\epsilon_p\frac{X^\bullet}{M^T}\right]H_\lambda\otimes e^\alpha\right)
    &=T_\lambda^{(0)}e_1\left[\epsilon_{-p}\frac{D^\dagger}{M}\right]\mathbb{E}_\lambda\otimes e^\alpha.
\end{align*}
By \cite[Corollary 4.13]{RW}, $\mathrm{span}\left\{\mathbb{E}_\mu\otimes e^\alpha\,|\,\core(\mu)=\alpha\right\}$ is closed under the action of any $f[D^\dagger]$ (we will also see this directly in the course of our computation).
Thus, for some coefficients $\mathfrak{d}_{\mu,\lambda}^{(p)}$,
\[
e_1\left[\epsilon_{-p}\frac{D^\dagger}{M}\right]\mathbb{E}_\lambda\otimes e^\alpha
=\sum_{\core(\mu)=\alpha}\mathfrak{d}_{\mu,\lambda}^{(p)}\mathbb{E}_\mu\otimes e^\alpha.
\]
We then have
\[
d_{\mu,\lambda}^{(p)}=\frac{T^{(0)}_\lambda}{T_\mu^{(0)}}\mathfrak{d}_{\mu,\lambda}=-\chi_0^{-1}\mathfrak{d}_{\mu,\lambda}.
\]

\subsubsection{Applying $\mathbb{D}$}
By Proposition \ref{DoperatorProperties}(4),
\[
-e_1\left[\epsilon_{-p}\frac{D^\dagger}{M}\right]
= \mathbb{D}_{\epsilon_0-\epsilon_p}.
\]
Using the explicit formulas for $\mathbb{D}$ and $\mathbb{E}_\lambda$, we follow similar steps as in the proof of \cite[Lemma 4.12]{RW} to obtain:
\begin{equation}
 \begin{aligned}
   \mathbb{D}_{\epsilon_0-\epsilon_p}\mathbb{E}_\lambda\otimes e^\alpha 
   &=(1-qt)^{r-1}\frac{z_0^2}{z_p}
   \prod_{i\in\ZZ/r\ZZ}\left(\frac{z_{i+1}}{z_{i+1}-tz_{i}}\right)
   \left(\frac{z_i}{z_{i}-qz_{i+1}}\right)
   \frac{\displaystyle \prod_{\square\in R_i(\lambda)}\left(z_i-qt\chi_\square\right)}
   {\displaystyle \prod_{\square\in A_i(\lambda)}\left(z_i-\chi_\square\right)}\\
   &\quad\times
   \Omega\left[\sum_{i\in\ZZ/r\ZZ}X^{(i)}\left(\left(\frac{1}{M_0}\right)^{(i)}-B_\lambda^{(i)} -z_i\right)\right]
   \otimes e^\alpha\Bigg|_{z^0}
\end{aligned}
\label{DSetup}
\end{equation}
where the rational functions are Laurent series expanded assuming
\begin{equation}
    \begin{aligned}
        |z_i|=1&& \text{ and } &&
        |q|,|t|<1.
    \end{aligned}
    \label{Expand}
\end{equation}

\subsubsection{Setup}
We will compute constant terms one variable at a time.
To do so, we will use the following lemma, which is just an application of partial fraction decomposition:
\begin{lem}[\protect{\cite[Lemma 4.1]{RW}}]\label{ConstTermLem}
When we are taking the constant term of an expression
\[
\frac{zF(z)}{(z-P_1)\cdots(z-P_k)}
\]
with respect to the variable $z$, where
\begin{itemize}
\item $F(z)$ is a series in nonnegative powers of $z$,
\item each pole is simple and nonzero, and
\item each pole is expanded as
\[
\frac{1}{z-P_i}=z^{-1}\sum_{n\ge 0}\frac{P_i^n}{z^n},
\]
\end{itemize}
the constant term is a sum over evaluations:
\begin{equation}
 \left. \frac{zF(z)}{(z-P_1)\cdots(z-P_k)}\right|_{z^0} =\sum_{i=1}^k \frac{F(P_i)}{\displaystyle\prod_{j\not=i}(P_i-P_j)}.   
 \label{CLem}
\end{equation}
\end{lem}

When computing the constant term iteratively, \eqref{Expand} may be insufficient to expand the evaluated poles.
To address this, we introduce variables $\ppp_q,\ppp_t$ such that
\begin{equation}
 |\ppp_q|,|\ppp_t|<|q^at^b|\hbox{ for all }(a,b)\in\ZZ_{\ge 0}^2.   
 \label{PCond}
\end{equation}
We then have
\begin{align}
\label{DE1}
   \mathbb{D}_{\epsilon_0-\epsilon_p}\mathbb{E}_\lambda\otimes e^\alpha 
   &=(1-qt)^{r-1}\frac{z_0^2}{z_p}
   \prod_{i\in\ZZ/r\ZZ}\left(\frac{z_{i+1}}{z_{i+1}-\ppp_tz_{i}}\right)
   \left(\frac{z_i}{z_{i}-\ppp_qz_{i+1}}\right)
   \frac{\displaystyle \prod_{\square\in R_i(\lambda)}\left(z_i-qt\chi_\square\right)}
   {\displaystyle \prod_{\square\in A_i(\lambda)}\left(z_i-\chi_\square\right)}\\
\label{DE2}
   &\quad\times
   \Omega\left[\sum_{i\in\ZZ/r\ZZ}X^{(i)}\left(\left(\frac{1}{M_0}\right)^{(i)}-B_\lambda^{(i)} -z_i\right)\right]
   \otimes e^\alpha\Bigg|_{z^0}\Bigg|_{(\ppp_q,\ppp_t)\mapsto (q,t)}.
\end{align}

\subsubsection{Evaluations from the constant term}
As we take constant terms using Lemma \ref{ConstTermLem}, each variable will specialize to a character, adding boxes to $B_\lambda^\bullet$ in \eqref{DE2}.
Note that
\[
\mathbb{E}_\mu=\Omega\left[-X^\bullet\frac{D_\mu^\bullet}{M}\right]
=\Omega\left[\sum_{i\in\ZZ/r\ZZ}X^{(i)}\left(\left(\frac{1}{M_0}\right)^{(i)}-B_\mu^{(i)}\right)\right]
\]
Thus, provided that these box additions to $\lambda$ yield a partition, we will indeed obtain some $\mathbb{E}_\mu$ with some coefficient.

Aside from $z_p$, every variable has at least a square in the numerator of \eqref{DE1} for which to implement Lemma \ref{ConstTermLem}.
On the other hand, $z_p$ has only a power of one, so we will begin with $z_p$ and proceed \textit{downward} in cyclic order.
There are three kinds of poles in \eqref{DE1} where $z_p$ is expanded in negative powers: 
\begin{itemize}
\item a \textit{$t$-pole}: $(z_p-\ppp_tz_{p-1})$,
\item a \textit{$q$-pole}: $(z_p-\ppp_qz_{p+1})$, or 
\item \textit{$\chi$-poles}: the poles $(z_p-\chi_\square)$ for some $\square\in A_p(\lambda)$.
\end{itemize}
If $z_p$ is evaluated at the $t$-pole, then note that upon sending $\ppp_t\mapsto t$, $z_{p-1}$ will not be evaluated at a removable box of some partition $\mu$.
Additionally, the $q$-pole of $z_{p-1}$ becomes unavailable:
\begin{equation}
 \frac{z_{p-1}}{z_{p-1}-\ppp_qz_p}\mapsto\frac{z_{p-1}}{z_{p-1}-\ppp_q\ppp_tz_{p-1}}=\frac{1}{1-\ppp_q\ppp_t}
 \label{TCancel}
\end{equation}
%Upon setting $(\ppp_q,\ppp_t)\mapsto (q,t)$, this will cancel with a $(1-qt)$ at start of \eqref{DE1}.
Similarly, if $z_{p}$ is evaluated at the $q$-pole, then $z_{p+1}$ will not be evaluated at a removable box and its $t$-pole disappears:
\begin{equation}
 \frac{z_{p+1}}{z_{p+1}-\ppp_tz_p}\mapsto\frac{z_{p+1}}{z_{p+1}-\ppp_q\ppp_tz_{p+1}}=\frac{1}{1-\ppp_q\ppp_t}   
\label{QCancel}
\end{equation}
%Again, there is cancellation with $(1-qt)$ when $(\ppp_q,\ppp_t)\mapsto (q,t)$.
Finally, the $\chi$-pole evaluation results in a box addition that yields another partition, and $z_p$ is sent to $\chi_p$.

Continuing downward in cyclic order, the options available to $z_i$ are similar.
One may be concerned that if $z_{i+1}$ was evaluated at a $t$-pole, then there would be poles of the form $(\ppp_t^k z_i-\chi_\square)$ for some $\square\in A_{i+k}(\lambda)$, but those are not expanded in negative powers of $z_i$ due to the enhanced conditions \eqref{PCond}.
Likewise, if $z_p$ was evaluated at the $q$-pole, then for the final variable $z_{p+1}$, we will not have to consider poles of the form $(\ppp_q^kz_{p+1}-\chi_\square)$ for some $\square\in A_{p+1-k}(\lambda)$.
Thus, accounting for the fact that if $z_{i+1}$ was evaluated at the $t$-pole, then $z_i$ no longer has a $q$-pole, we see that variables are evaluated into characters of their corresponding color, and they can be grouped into
\begin{itemize}
    \item vertical chains with the bottommost box given by an addable box of $\lambda$ and
    \item horizontal chains with the leftmost box given by an addable box of $\lambda$.
\end{itemize}
Because we are only adding one box per color, for our box additions to result in a partition $\mu$, we only need to worry about passing through the box with character $qt\chi_\square$ for some $\square\in R_j(\lambda)$.
This is precisely excluded by the zero $(z_j-qt\chi_\square)$ in \eqref{DE1}.
The only case where this evaluation is allowed is when $z_{j-1}$ and $z_{j+1}$ are evaluated at the boxes below and to the left of the box with character $qt\chi_\square$, respectively, i.e. our box additions round the corner.
In this case, $(z_j-qt\chi_\square)$ cancels with $(z_j-\ppp_qz_{j-1})$.
Thus, the end result for some string of evaluations is a multiple of $\mathbb{E}_\mu$ with $\lambda \subset_1 \mu$, and $z_i\mapsto\chi_i$.

\subsubsection{Calculation of coefficients}
Lastly, we pin down the coefficient $\mathfrak{d}_{\mu,\lambda}^{(p)}$.
Our main task here is to keep track of what happens to the factors
\[
\frac{z_i}{(z_i-\ppp_qz_{i+1})(z_i-\ppp_tz_{i-1})}.
\]
This depends on the adjacency among the three boxes $\square_{i+1}$, $\square_i$, and $\square_{i-1}$.
There are a few cases:
\begin{enumerate}
    \item \textit{All three are horizontally adjacent}: $(z_i-\ppp_qz_{i+1})$ is removed when taking the constant term and the rest leave behind $(1-qt)^{-1}$ as in \eqref{QCancel}.
    This will cancel with $(1-qt)$ in \eqref{DE1}.
    We note that $\square_i$ is not removable.
    \item \textit{All three are vertically adjacent}: this is similar to the previous case.
    \item \textit{Contiguous, forming an L}: $z_i$ was evaluated at a $\chi$-pole.
    The factors $z_i/(z_i-\ppp_tz_{i-1})$ become $(1-qt)^{-1}$ as in \eqref{QCancel} and cancels with a $(1-qt)$ from \eqref{DE1}.
    The remaining $(z_i-\ppp_qz_{i+1})$ becomes $\chi_i-\chi_\square$ for a new $\square\in A_i(\mu)$.
    Note that $\square_i$ is not removable.
    \item \textit{Contiguous, forming an \rotatebox[origin=c]{180}{L}}: $\square_i$ is diagonally adjacent to a removable corner of $\lambda$.
    Then $(z_i-\ppp_tz_{i-1})$ is removed when taking the constant term and $(z_i-\ppp_qz_{i+1})$ cancels out with $(z_i-qt\chi_\square)$ for $\square\in R_i(\lambda)$.
    Note that $\square_i\in R_i(\mu)$.
    We take the remaining $z_i$ and combine it with $(1-qt)$ in \eqref{DE1} to form $(\chi_i-qt\chi_i)$.
    \item \textit{$\square_i$ and $\square_{i+1}$ are vertically adjacent, $\square_{i-1}$ is disjoint}: $z_i$ is evaluated at a $\chi$-pole, for otherwise $\mu$ would not be a partition.
    The factors $z_i/(z_i-\ppp_qz_{i+1})$ become $(1-qt)^{-1}$ as in \eqref{TCancel} and cancels with a $(1-qt)$ from \eqref{DE1}.
    Note that $\square_i$ is not removable.
    Either the box above $\square_{i-1}$ is a new addable box of $\mu$ or the box left of $\square_{i-1}$ is a removable box of $\lambda$ that is no longer removable in $\mu$.
    Correspondingly, $(z_i-\ppp_tz_{i-1})$ creates or cancels the corresponding factor.
    \item \textit{$\square_i$ and $\square_{i+1}$ are horizontally adjacent, $\square_{i-1}$ is  disjoint}: $(z_i-\ppp_q z_{i+1})$ is removed when taking the constant term and $(z_i-\ppp_t z_{i-1})$ is as in the previous case.
    Note that $\square_i\in R_i(\mu)$ so we combine  $z_i$ and a $(1-qt)$ from \eqref{DE1} to form $(\chi_i-qt\chi_i)$.
    \item \textit{$\square_i$ and $\square_{i-1}$ are adjacent, $\square_{i+1}$ is disjoint}: this similar to the previous two cases.
    \item \textit{All three boxes are disjoint}: we handle both boxes like we did for $\square_{i-1}$ in (6).
    Note that $\square_i\in R_i(\mu)$, so we form a $(\chi_i-qt\chi_i)$.
\end{enumerate}
Altogether, we have
\begin{align}
\nonumber
 -\ddd_{\mu,\lambda}^{(p)}
 &=\frac{1}{1-qt}\left(\frac{\chi_0^2}{\chi_p}\right)
\prod_{i\in\ZZ/r\ZZ}
\frac{\displaystyle\prod_{\square\in R_i(\mu)}(\chi_i-qt\chi_\square)}
{\displaystyle\prod_{\square\in A_i(\mu)}(\chi_i-\chi_\square)}\\
 &=\frac{1}{qt-1}\left(\frac{\chi_0^2}{\chi_p}\right)
\prod_{i\in\ZZ/r\ZZ}
\frac{\displaystyle\prod_{\square\in R_i(\mu)}\left(1-q^{-1}t^{-1}\frac{\chi_i}{\chi_\square}\right)}
{\displaystyle\prod_{\square\in A_i(\mu)}\left(1-\frac{\chi_i}{\chi_\square}\right)}.
%&=\frac{1}{1-qt}\left(\frac{\chi_0}{\chi_p}\right)
%\prod_{i\in\ZZ/r\ZZ}
%\chi_i^{\#R_i(\mu)-\#A_i(\mu)+\delta_{i,0}}
%\frac{\displaystyle\prod_{\square\in R_i(\mu)}\left(1-qt\frac{\chi_\square}{\chi_i}\right)}
%{\displaystyle\prod_{\square\in A_i(\mu)}\left(1-\frac{\chi_\square}{\chi_i}\right)}.
\label{PieriLast}
\end{align}
For \eqref{PieriLast}, we use
\begin{equation}
\prod_{i\in \ZZ/r\ZZ}\frac{\displaystyle\prod_{\square\in A_i(\lambda)}(-\chi_\square)}
{\displaystyle\prod_{\square\in R_i(\lambda)}(-qt\chi_\square)}=-1
\label{AddRem}
\end{equation}
(cf. the proof of \cite[Theorem 2.2]{GarsiaTesler}).
We obtain \eqref{E1} by 
%replacing
%\[
%-C_i=\#R_i(\mu)-\#A_i(\mu)+\delta_{i,0}
%\]
%and 
dividing by $-\chi_0$ to obtain $d_{\mu,\lambda}^{(p)}$.

\subsection{The dual Pieri rule}
We will access $h_1^\perp[X^{(p)}]$ via adjunction with respect to $\langle -, -\rangle_{\ast}$.
Specifically, the adjoint to $h_1^\perp[X^{(p)}]$ is $-e_1\left[\epsilon_{-p}\frac{X^\bullet}{M^T}\right]$.
Observe, however, that we cannot directly use the Pieri coefficients \eqref{E1} with respect to the basis $\{H_\lambda\}$, since the adjoint will now act on the the dual basis $\{H_\lambda^\dagger\}$.

\subsubsection{Applying $\downarrow$}

Define $d_{\mu,\lambda}^{\dagger(p)}$ by
\begin{equation}
 -e_1\left[\epsilon_{-p}\frac{X^\bullet}{M^T}\right] H_\lambda^\dagger=
\sum_{\lambda\subset_1\mu}  d_{\mu,\lambda}^{\dagger(p)}H_\mu^\dagger .  
\label{DDagger}
\end{equation}

\begin{lem}\label{PieriDown}
    The coefficient $d_{\mu,\lambda}^{\dagger(p)}$ is given by
    \begin{align}
d_{\mu,\lambda}^{\dagger(p)}    &=\frac{1}{1-qt}
\left(\frac{\chi_p}{\chi_0}\right)
\prod_{i\in\ZZ/r\ZZ} 
\frac{\displaystyle\prod_{\square\in R_i(\mu)}\left( 1-qt\frac{\chi_\square}{\chi_i} \right)}
{\displaystyle\prod_{\square\in A_i(\mu)}\left(1-\frac{\chi_\square}{\chi_i}\right)}.       
\label{DDaggerForm}
    \end{align}
\end{lem}

\begin{proof}
Note that
\begin{align*}
    \downarrow H_\lambda^\dagger&= H_\lambda,\\
    \downarrow \left(-e_1\left[\epsilon_{-p}\frac{X^\bullet}{M^T}\right]\right)&= qte_1\left[\epsilon_p\frac{X^\bullet}{M^T}\right].
\end{align*}
Applying $\downarrow$ to both sides of \eqref{DDagger}, we have
\[
\left(d_{\mu,\lambda}^{\dagger(p)}\right)\Bigg|_{(q,t)\mapsto (q^{-1},t^{-1})}
=qt d_{\mu,\lambda}^{(p)}.\qedhere
\]
%Thus, we have
%\[
%     d_{\mu,\lambda}^{\dagger(p)}
%     =
%    \frac{1}{qt-1}\chi_p
%\prod_{i\in\ZZ/r\ZZ} \chi_i^{C_i}
%\frac{\displaystyle\prod_{\square\in R_i(\mu)}\left( 1-q^{-1}t^{-1}\frac{\chi_i}{\chi_\square} \right)}
%{\displaystyle\prod_{\square\in A_i(\mu)}\left(1-\frac{\chi_i}{\chi_\square}\right)}.
%\]
%To get to (\ref{DDaggerForm}), we rebalance the binomials.
%First, pulling out every instance of $\chi_i$ yields 
%\[\chi_i^{\#R_i(\lambda)-\#A_i(\lambda)}= \chi_i^{-\delta_{i,0}-C_i}.\]
%Finally, note that
\end{proof}

\subsubsection{Applying adjunction}
Our statement about taking adjoints is summarized by the following:
\begin{lem}\label{PieriAd}
    We have
    \[
    c_{\lambda,\mu}^{(p)}= \frac{N_\mu}{N_\lambda} d_{\mu,\lambda}^{\dagger(p)}.
    \]
\end{lem}

\begin{proof}
    We can pick out $c_{\lambda,\mu}^{(p)}$ with the following inner product:
    \begin{align*}
        c_{\lambda,\mu}^{(p)}
        &=\frac{1}{N_\lambda}\left\langle H^\dagger_\lambda, h_1^\perp[X^{(p)}]H_\mu\right\rangle_\ast\\
        &=\frac{1}{N_\lambda}\left\langle -e_1\left[\epsilon_{-p}\frac{X^\bullet}{M^T}\right] H_\lambda^\dagger, H_\mu\right\rangle_\ast\\
        &=\frac{N_\mu}{N_\lambda} d_{\mu,\lambda}^{\dagger(p)}.\qedhere
    \end{align*}
\end{proof}

\subsubsection{Consolidating hooks}
The ratio $N_\mu/N_\lambda$ can be unpacked using the following lemma:
\begin{lem}\label{HookLem}
Let $\blacksquare\in A_i(\lambda)$.
For $r>1$, we have
\begin{align}
\label{NCancel}
\left(\frac{N_{\lambda\cup\blacksquare}}{N_\lambda}\right)\frac{\displaystyle\prod_{\square\in R_i(\lambda\cup\blacksquare)}\left( 1-qt\frac{\chi_\square}{\chi_\blacksquare} \right)}{\displaystyle\prod_{\square\in A_i(\lambda\cup\blacksquare)}\left( 1-\frac{\chi_\square}{\chi_\blacksquare} \right)}
&=
\frac{\displaystyle\prod_{\square\in A_i(\lambda)}\left( 1-qt\frac{\chi_\blacksquare}{\chi_\square} \right)}{\displaystyle\prod_{\square\in R_i(\lambda)}\left( 1-\frac{\chi_\blacksquare}{\chi_\square} \right)}.
\end{align}
\end{lem}

\begin{proof}
We can rewrite the left-hand-side of (\ref{NCancel}) as
\begin{equation}
(1-qt)\left(\frac{N_{\lambda\cup\blacksquare}}{N_\lambda}\right)\frac{\displaystyle\prod_{\square\in R_i(\lambda)}\left( 1-qt\frac{\chi_\square}{\chi_\blacksquare} \right)}{\displaystyle\prod_{\substack{\square\in A_i(\lambda)\\\square\not=\blacksquare}}\left( 1-\frac{\chi_\square}{\chi_\blacksquare} \right)}
\label{NCancel2}
\end{equation}
(note that $r>1$ is important here).
Recall that
\begin{equation}
N_\lambda=\prod_{\substack{\square\in\lambda\\ h(\square)\equiv 0\hbox{ \footnotesize{mod} }r}}
   \left(1-q^{a_\lambda(\square)+1}t^{-l_\lambda(\square)}\right)
   \left(1-q^{-a_\lambda(\square)}t^{l_\lambda(\square)+1}\right).
\label{NormForm}
\end{equation}
For a box $\square \in \lambda$, let $\square_N$ be the northernmost box in the same column as $\square$. Let $\square'$ be the cell directly east, and ${}'\square$ the cell directly west.
Consider a hook corresponding to $\square \in \lambda$ in the same row as $\blacksquare$-- the hook starts at $\square_N$, moves down to $\square$, then moves east to $\blacksquare$.
\begin{itemize}
    \item If $h_\lambda(\square) \equiv 0 \mod r$ and $\square_N$ is not removable, then the factor in (\ref{NormForm}) cancels with the factor in $N_{\lambda \cup \blacksquare}$ corresponding to $\square'$.
    \item If $h_{\lambda \cup \blacksquare}(\square) \equiv 0 \mod r$, and $\square_N$ is not below and addable corner, then its corresponding factor cancels with the factor for ${}'\square$ in $N_{\lambda}$.
    \item If $h_\lambda(\square) \equiv 0 \mod r$ and $\square_N$ is removable, then we must have $\square_N \in R_i(\lambda)$.
    Furthermore, 
    \begin{align*}
    q^{-a_\lambda(\square)}t^{l_\lambda(\square)+1} = qt \chi_{\square_N}/\chi_{\blacksquare} && \text{ and } &&
    q^{a_\lambda(\square)+1}t^{-l_\lambda(\square)} = \chi_\blacksquare/\chi_{\square_N}.
    \end{align*}
    \item If $h_{\lambda \cup \blacksquare}(\square) \equiv 0 \mod r$ and $\square_N$ is below an addable corner $\overline{\square}_N$, then $\overline{\square}_N \in A_i(\lambda)$.
    Furthermore, 
    \begin{align*}
    q^{-a_{\lambda \cup \blacksquare}(\square)}t^{l_{\lambda \cup \blacksquare}(\square)+1} =  \chi_{\overline{\square}_N}/\chi_{\blacksquare} && \text{ and } &&
    q^{a_{\lambda \cup \blacksquare}(\square)+1}t^{-l_{\lambda \cup \blacksquare}(\square)} = qt\chi_\blacksquare/\chi_{\overline{\square}_N}.
    \end{align*}
\end{itemize}
A similar argument analyzes the hooks which start with $\blacksquare$ above $\square$. We see that after cancelling terms that appear, in (\ref{NCancel2}), one gets
\[
(1-qt)\frac{\displaystyle\prod_{\substack{\square\in A_i(\lambda)\\\square\not=\blacksquare}}\left( 1-qt\frac{\chi_\blacksquare}{\chi_\square} \right)}{\displaystyle\prod_{\square\in R_i(\lambda)}\left( 1-\frac{\chi_\blacksquare}{\chi_\square} \right)},
\]
which is the right-hand side of (\ref{NCancel}).
\end{proof}

\subsubsection{Proof of (\ref{H1})}
By Lemma \ref{PieriAd}, we need to compute
\[
c_{\lambda,\mu}^{(p)}=\frac{N_\mu}{N_\lambda}d_{\mu,\lambda}^{\dagger(p)}.
\]
We will show that
\begin{equation}
c_{\lambda,\mu}^{(p)}
=
\frac{1}{1-qt}
\left(\frac{\chi_p}{\chi_0}\right)
\prod_{i\in\ZZ/r\ZZ} 
\frac{\displaystyle\prod_{\square\in A_i(\lambda)}  \left(1-qt\frac{\chi_i}{\chi_\square}\right)}
{\displaystyle\prod_{\square\in R_i(\lambda)} \left( 1-\frac{\chi_i}{\chi_\square} \right)}.
\label{H1Almost}
\end{equation}
%The equality between the right-hand-side of (\ref{H1Almost}) and that of (\ref{H1}) follows from
%\begin{align*}
%\chi^{\#A_i(\lambda)-\#R_i(\lambda)}&=\chi^{C_i+\delta_{i,0}}\\
%\prod_{i\in\ZZ/r\ZZ}\frac{\displaystyle\prod_{R_i(\lambda)}(-\chi_\square)}{\displaystyle\prod_{A_i(\lambda)}(-q^{-1}t^{-1}\chi_\square)}&=-qt
%\end{align*}
%(cf. the proof of Lemma \ref{PieriDown}).

Recall that the index $i$ in $\square_i\in\mu/\lambda$ corresponds to its color.
Let us order the boxes $\square_{i_1}, \ldots, \square_{i_r}$ such that for all $k=1,\ldots, r$,
\[
\lambda(k)\coloneq \mu/\{\square_{i_1},\ldots, \square_{i_k}\}
\]
is a partition.
We will set $\lambda(0)=\mu$.
Thus, we have
\[
\mu=\lambda(0)\supset\lambda(1)\supset\cdots\supset\lambda(r)=\lambda.
\]
Our strategy will be to factor the coefficient
\[
c_{\lambda,\mu}^{(p)}=\frac{N_\mu}{N_\lambda}d_{\mu,\lambda}^{\dagger(p)}=d_{\mu,\lambda}^{\dagger(p)}\prod_{k=1}^{r}\frac{N_{\lambda(k-1)}}{N_{\lambda(k)}}
\]
and then apply Lemma \ref{HookLem} to $N_{\lambda(k-1)}/N_{\lambda(k)}$ and the binomials in $d_{\mu,\lambda}^{\dagger(p)}$ corresponding to $\chi_{i_k}$.
However, the factors in $d_{\mu,\lambda}^{\dagger(p)}$ are in terms of $R_{i_k}(\mu)$ and $A_{i_k}(\mu)$, not $R_{i_k}(\lambda(k-1))$ and $A_{i_k}(\lambda(k-1))$.
Additionally, even if we can apply Lemma \ref{HookLem}, the resulting factors are in terms of $R_{i_k}(\lambda(k))$ and $A_{i_k}(\lambda(k))$, not $R_{i_k}(\lambda)$ and $A_{i_k}(\lambda)$.

It turns out that both problems can be resolved simultaneously.
Note that we can immediately apply Lemma \ref{HookLem} at step $k=1$.
By induction, at step $k$, applying Lemma \ref{HookLem} yields
\begin{align}
\nonumber
&\left(\frac{N_{\lambda(k-1)}}{N_{\lambda(k)}}\right)
\prod_{\ell=k}^r
\frac{\displaystyle\prod_{\square\in R_{i_\ell}(\lambda(k-1))}\left( 1-qt\frac{\chi_\square}{\chi_{i_\ell}} \right)}{\displaystyle\prod_{\square\in A_{i_\ell}(\lambda(k-1))}\left( 1-\frac{\chi_\square}{\chi_{i_\ell}} \right)}\\
&=
\underbrace{\frac{\displaystyle\prod_{\square\in A_{i_k}(\lambda(k))}\left( 1-qt\frac{\chi_{i_k}}{\chi_\square} \right)}{\displaystyle\prod_{\square\in R_{i_k}(\lambda(k))}\left( 1-\frac{\chi_{i_k}}{\chi_\square} \right)}}_{(a)}
\,
\overbrace{\prod_{\ell=k+1}^r
\frac{\displaystyle\prod_{\square\in R_{i_\ell}(\lambda(k-1))}\left( 1-qt\frac{\chi_\square}{\chi_{i_\ell}} \right)}{\displaystyle\prod_{\square\in A_{i_\ell}(\lambda(k-1))}\left( 1-\frac{\chi_\square}{\chi_{i_\ell}} \right)}}^{(b)}.
\label{PieriInduct0}
\end{align}
We need to show that (\ref{PieriInduct0}) is equal to:
\begin{equation}
\underbrace{
\frac{\displaystyle\prod_{\square\in A_{i_k}(\lambda)}\left( 1-qt\frac{\chi_{i_k}}{\chi_\square} \right)}{\displaystyle\prod_{\square\in R_{i_k}(\lambda)}\left( 1-\frac{\chi_{i_k}}{\chi_\square} \right)}}_{(a')}
\,
\overbrace{
\prod_{\ell=k+1}^r
\frac{\displaystyle\prod_{\square\in R_{i_\ell}(\lambda(k))}\left( 1-qt\frac{\chi_\square}{\chi_{i_\ell}} \right)}{\displaystyle\prod_{\square\in A_{i_\ell}(\lambda(k))}\left( 1-\frac{\chi_\square}{\chi_{i_\ell}} \right)}}^{(b')}.
\label{PieriInduct}
\end{equation}
The factors $(a)$ and $(a')$ can only differ if $\square_{i_{k}+1}$ and/or $\square_{i_{k}-1}$ are in $\lambda(k)$.
If this is the case, we will say that $i_{k}+1$ and/or $i_k-1$ are \textit{after} $i_k$.
Likewise, $(b)$ and $(b')$ will only differ if $i_k+1$ and/or $i_k-1$ are after $i_k$.
Let $(a)/(a')$ and $(b)/(b')$ denote the corresponding quotients; we will show that
\[
\frac{(a)}{(a')}
\cdot \frac{(b)}{(b')} =1.
\]
The analysis requires us to consider many cases:
\begin{enumerate}
\item \textit{Only $i_k+1$ is after $i_k$ with $\square_{i_k}$ and $\square_{i_k+1}$ adjacent}: we have $(a)/(a')=(1-qt)$ due to $\square_{i_k}\in A_{i_k}(\lambda(k))$ but $\square_{i_k}\not\in A_{i_k}(\lambda)$, and $(b)/(b')=(1-qt)^{-1}$ due to $\square_{i_k+1}\in R_{i_k+1}(\lambda(k))$ but $\square_{i_k+1}\not\in R_{i_k+1}(\lambda(k-1))$.
\item \textit{Only $i_k+1$ is after $i_k$ with $\square_{i_k}$ and $\square_{i_k+1}$ disjoint}: here, either the square right of $\square_{i_k+1}$ is in $A_{i_k}(\lambda(k))$ but not in $A_{i_k}(\lambda)$ or the square below $\square_{i_k+1}$ is in $R_{i_k}(\lambda)$ but not in $R_{i_k}(\lambda(k))$.
In either case, $(a)/(a')=(1-t\chi_{i_k}/\chi_{i_k+1})$.
Likewise, either the box above $\square_{i_k}$ is in $A_{i_k+1}(\lambda(k-1))$ but not in $A_{i_k+1}(\lambda(k))$ or the box left of $\square_{i_k}$ is in $R_{i_k+1}(\lambda(k))$ but not in $R_{i_k+1}(\lambda(k-1))$.
Regardless, $(b)/(b')=(1-t\chi_{i_k}/\chi_{i_k+1})^{-1}$.
\item \textit{Only $i_k-1$ is after $i_k$}: this case is treated like cases (1) and (2).
\item \textit{Both $i_k+1$ and $i_k-1$ are after $i_k$ with only one of $\{\square_{i_k+1},\square_{i_k-1}\}$ adjacent to $\square_{i_k}$}: we treat the adjacent one as in case (1) and the disjoint one as in case (2).
\item \textit{Both $i_k+1$ and $i_k-1$ after $i_k$ with all boxes disjoint}: we treat both both $\square_{i_k+1}$ and $\square_{i_k-1}$ as in case (2).
\item \textit{Both $i_k+1$ and $i_k-1$ are after $i_k$ and all boxes are adjacent}: here, the boxes $\square_{i_k+1}$, $\square_{i_k}$, and $\square_{i_k-1}$ must form a \rotatebox[origin=c]{180}{L} shape.
The removable of all three yields a box in $R_{i_k}(\lambda)$ that is not in $A_{i_k}(\lambda(k))$, and as in case (1), $\square_{i_k}\in A_{i_k}(\lambda(k))$ but not in $A_{i_k}(\lambda)$.
In total, $(a)/(a')=(1-qt)^2$.
On the other hand, as in $(1)$, $\square_{i_k+1}$ and $\square_{i_k-1}$ yield a factor of $(1-qt)^{-1}$, giving $(b)/(b')=(1-qt)^{-2}$.\qed
\end{enumerate}

\section*{Acknowledgements}
We'd like to thank 
Hunter Dinkins for
developing and sharing a MAPLE package that allowed us to check our Pieri formulas. We would also like to thank Sean Griffin.

This work was supported by ERC consolidator grant No. 101001159 ``Refined invariants in combinatorics, low-dimensional topology and geometry of moduli spaces''.

\bibliographystyle{alpha}
\bibliography{Wreath}

\end{document}